\documentclass[a4paper]{article}

\usepackage[backend=biber,style=alphabetic,firstinits=true]{biblatex}
\usepackage[utf8]{inputenc}

\usepackage{amssymb,amsmath,amsthm}
\usepackage[english]{babel}

\usepackage{todonotes}
\usepackage{verbatim} 
\usepackage{enumerate}
\usepackage{mathtools}

\usepackage{a4}
\usepackage{multirow}
\usepackage[footnotesize,bf,centerlast]{caption}
\usepackage{xcolor,colortbl}
\definecolor{Gray}{gray}{0.80}
\definecolor{LightGray}{gray}{0.90}

\usepackage[normalem]{ulem}

\usepackage{bbm}

\newcommand{\cA}{\mathcal{A}}

\newcommand{\cC}{\mathcal{C}}
\newcommand{\cD}{\mathcal{D}}

\newcommand{\cG}{\mathcal{G}}

\newcommand{\cL}{\mathcal{L}}

\newcommand{\cP}{\mathcal{P}}

\newcommand{\bE}{\mathbb{E}}

\newcommand{\bN}{\mathbb{N}}

\newcommand{\bR}{\mathbb{R}}

\newcommand{\PR}{\mathbb{P}}

\newcommand{\dd}{ \mathrm{d}}

\renewcommand{\epsilon}{\varepsilon}

\newcommand{\vn}[1]{\left| \! \left| #1\right| \! \right|}

\newcommand{\ip}[2]{\langle #1,#2\rangle}

\numberwithin{equation}{section}

\newtheorem{theorem}{Theorem}[section]
\newtheorem{lemma}[theorem]{Lemma}
\newtheorem{proposition}[theorem]{Proposition}

\theoremstyle{definition}
\newtheorem{definition}[theorem]{Definition}

\newtheorem{remark}[theorem]{Remark}

\newtheorem{assumption}[theorem]{Assumption}

\newcommand{\Ftodo}[1]{\todo[size=\scriptsize,color=green!30]{#1}}

\begin{document}

\title{Dynamical moderate deviations for the Curie-Weiss model}

\author{
\renewcommand{\thefootnote}{\arabic{footnote}}
Francesca Collet\footnotemark[1] \, and Richard C. Kraaij\footnotemark[2]
}

\footnotetext[1]{
	Delft Institute of Applied Mathematics,	Delft University of Technology,
	Mekelweg 4, 2628 CD Delft, The Netherlands,	E-mail: \texttt{f.collet-1@tudelft.nl}.
}

\footnotetext[2]{
Fakultät für Mathematik, Ruhr-University of Bochum, Postfach 102148, 
44721 Bochum, Germany, E-mail: \texttt{richard.kraaij@rub.de}.
}

\maketitle

\begin{abstract}
We derive moderate deviation principles for the trajectory of the empirical magnetization of the standard Curie-Weiss model via a general analytic approach based on convergence of generators and uniqueness of viscosity solutions for associated Hamilton-Jacobi equations. The moderate asymptotics depend crucially on the phase under consideration. \\

\end{abstract}

%===============================================================================
\section{Introduction}
%===============================================================================

The study of the normalized sum of random variables and its asymptotic behavior plays a central role in probability and statistical mechanics. Whenever the variables are independent and have finite variance, the central limit theorem ensures that the sum with square-root normalization converges to a Gaussian distribution. The generalization of this result to dependent variables is particularly interesting in statistical mechanics where the random variables are correlated through an interaction Hamiltonian. 
Ellis and Newman characterized the distribution of the normalized sum of spins (\emph{empirical magnetization}) for a wide class of mean-field Hamiltonian of Curie-Weiss type \cite{ElNe78a,ElNe78b,ElNeRo80}. They found conditions, in terms of thermodynamic properties, that lead in the infinite volume limit to a Gaussian behavior and those which lead to a higher order exponential probability distribution. 

A natural further step was the investigation of large and moderate fluctuations for the magnetization. The large deviation principle is due to Ellis \cite{Ell85}. Moderate deviation properties have been treated by Eichelsbacher and L{\"o}we in \cite{EiLo04}. 
A moderate deviation principle is technically a large deviation principle and consists in a refinement of a (standard or non-standard) central limit theorem, in the sense that it characterizes the exponential decay of deviations from the average on a smaller scale. In \cite{EiLo04}, it was shown that the physical phase transition in Curie-Weiss type models is reflected by a radical change in the asymptotic behavior of moderate deviations. Indeed, whereas the rate function is quadratic at non-critical temperatures, it becomes non-quadratic at criticality. 

All the results mentioned so far have been derived at equilibrium; on the contrary, we are interested in describing the time evolution of fluctuations, obtaining non-equilibrium properties. Fluctuations for the standard Curie-Weiss model were studied on the level of a path-space large deviation principle by Comets \cite{Co89} and Kraaij \cite{Kr16b} and on the level of a path-space central limit theorem by Collet and Dai Pra in \cite{CoDaP12}. The purpose of the present paper is to study dynamical moderate deviations to complete the analysis of fluctuations of the empirical magnetization.  We apply the generator convergence approach to  large deviations by Feng-Kurtz \cite{FK06} to characterize the most likely behaviour for the trajectories of fluctuations around the stationary solution(s) in the various regimes. The moderate asymptotics depend crucially on the phase we are considering. The criticality of the inverse temperature $\beta=1$ shows up at this level via a sudden change in the speed and rate function of the moderate deviation principle for the magnetization. In particular, our findings indicate that fluctuations are Gaussian-like in the sub- and super-critical regimes, while they are not at the critical point. \\
Besides, we analyze the deviation behaviour when the temperature is size-dependent and is increasing to the critical point. In this case, the rate function inherits features of both the uniqueness and multiple phases: it is the combination of the critical and non-critical rate functions. To conclude, it is worth to mention that our statements are in agreement with the results found in \cite{EiLo04}.\\

The outline of the paper is as follows: in Section~\ref{sct:results} we formally introduce the Curie-Weiss model and we state our main results. All the proofs, if not immediate, are postponed to Section~\ref{sct:proofs}. Appendix~\ref{sct:app:LDPviaHJequation} is devoted to the derivation of a large deviation principle via solution of Hamilton-Jacobi equation and it is included to make the paper as much self-contained as possible.

%===============================================================================
\section{Model and main results}\label{sct:results}
%===============================================================================

\subsection{Notation and definitions}
	
	Before we give our main results, we introduce some notation. We start with the definition of good rate-functions and what it means for random variables to satisfy a large deviation principle.
	
\begin{definition}
	Let $X_1,X_2,\dots$ be random variables on a Polish space $F$. Furthermore let $I : F \rightarrow [0,\infty]$.
	\begin{enumerate}
		\item We say that $I$ is a \textit{good rate-function} if for every $c \geq 0$, the set $\{x \, | \, I(x) \leq c\}$ is compact.
		\item We say that the sequence $\{X_n\}_{n\geq 1}$ is \textit{exponentially tight} if for every $a \geq 0$ there is a compact set $K_a \subseteq X$ such that $\limsup_n \PR[X \in K^c_a] \leq - a$.
		\item We say that the sequence $\{X_n\}_{n\geq 1}$ satisfies the \textit{large deviation principle} with rate $r(n)$ and good rate-function $I$, denoted by 
		\begin{equation*}
		\PR[X_n \approx a] \sim e^{-r(n) I(a)},
		\end{equation*}
		if we have for every closed set $A \subseteq X$
		\begin{equation*}
		\limsup_{n \rightarrow \infty} \frac{1}{r(n)} \log \PR[X_n \in A] \leq - \inf_{x \in A} I(x),
		\end{equation*}
		and for every open set $U \subseteq X$ 
		\begin{equation*}
		\liminf_{n \rightarrow \infty} \frac{1}{r(n)} \log \PR[X_n \in U] \geq - \inf_{x \in U} I(x).
		\end{equation*}
	\end{enumerate}
\end{definition}
	
Throughout the whole paper $\cA\cC$ will denote the set of absolutely continuous curves in $\bR$. 

\begin{definition} 
A curve $\gamma: [0,T] \to \mathbb{R}$ is absolutely continuous if there exists a function $g \in L^1[0,T]$ such that for $t \in [0,T]$ we have $\gamma(t) = \gamma(0) + \int_0^t g(s) \dd s$. We write $g = \dot{\gamma}$.

A curve $\gamma: \bR^+ \to \mathbb{R}$ is absolutely continuous if the restriction to $[0,T]$ is absolutely continuous for every $T \geq 0$.
\end{definition}

\subsection{Glauber dynamics for the Curie-Weiss model}

Let $\sigma = \left( \sigma_i \right)_{i=1}^n \in \{-1,+1\}^n$ be a configuration of $n$ spins and denote by 
\begin{equation*}
m_n(\sigma) = n^{-1} \sum_{i=1}^n \sigma_i
\end{equation*}
the empirical magnetization. The stochastic process $\{\sigma(t)\}_{t \geq 0}$ is described as follows. For $\sigma \in \{-1,+1\}^n$, let us define $\sigma^j$ the configuration obtained from $\sigma$ by flipping the $j$-th spin. The spins will be assumed to evolve with Glauber spin-flip dynamics: at any time $t$, the system may experience a transition $\sigma \to \sigma^j$ at rate $\exp \{ - \beta \sigma_j(t) m_n(t)\}$, where $\beta > 0$ represents the inverse temperature and where by abuse of notation $m_n(t) := m_n(\sigma(t))$. More formally, we can say that $\{ \sigma(t) \}_{t \geq 0}$ is a Markov process on $\{-1,+1\}^n$, with infinitesimal generator

\begin{equation}\label{CW:micro:gen}
\mathcal{G}_n f (\sigma) = \sum_{i=1}^n e^{- \beta \sigma_i m_n(\sigma)} \left[ f \left( \sigma^i \right) - f(\sigma)\right].
\end{equation}
Let
\[
E_n := m_n \left( \{-1,+1\}^n \right) = \left\{ -1, -1+ \frac{2}{n}, \dots, 1 - \frac{2}{n}, 1 \right\} \subseteq [-1,1]
\]
be the set of possible values taken by the magnetization (we will keep using this notation for the state space of the magnetization). The Glauber dynamics \eqref{CW:micro:gen} on the configurations induce Markovian dynamics for the process $\{ m_n(t) \}_{t \geq 0}$ on $E_n$, that in turn evolves with generator
\begin{equation*}
\cA_nf(x) = n \frac{1-x}{2} e^{\beta x} \left[f\left(x + \frac{2}{n}\right) - f(x)\right] + n \frac{1+x}{2} e^{-\beta x} \left[f\left(x - \frac{2}{n}\right) - f(x)\right].
\end{equation*}

This generator can be derived in two ways from \eqref{CW:micro:gen}. First of all, the microscopic jumps induce a change of size $\frac{2}{n}$ on the empirical magnetization. The jump rate of $x$ to $x + \frac{2}{n}$ corresponds to any $-1$ spin switching to $+1$ with rate $e^{\beta x}$. The total number of $-1$ spins can be computed from the empirical magnetization $x$ and equals $n\frac{1-x}{2}$. A similar computation yields the jump rate of $x$ to $x - \frac{2}{n}$. A second way to see that $\cA_n$ is the generator of the empirical magnetization is via the martingale problem and the property that $\cA_n f(m_n(\sigma)) = \cG_n(f \circ m_n)(\sigma)$. 

Assume the initial condition $m_n(0)$ obeys a large deviation principle, then it can be shown that $\{m_n(t)\}_{t \geq 0}$ obeys a large deviation principle on the Skorohod space of c{\`a}dl{\`a}g functions $D_\bR(\bR^+)$. We refer to \cite{EK86} for definition and properties of Skorohod spaces and to \cite{Co89,DuRaWu16} for the proof of the large deviation principle. Moreover, see \cite[Theorem 1]{Kr16b} for a LDP obtained by using similar techniques as in this paper. This path space large deviation principle allows to derive the infinite volume dynamics for our model: if $m_n(0)$ converges weakly to the constant  $m_0$, then the empirical magnetization process $\{m_n(t)\}_{t \geq 0}$ converges weakly, as $n \to \infty$, to the solution of 
\begin{equation}\label{CW:macro:dyn}
\dot{m}(t) = - 2 \, m(t) \cosh (\beta m(t)) + 2 \sinh (\beta m(t)) 
\end{equation}
with initial condition $m_0$. It is well known that the dynamical system \eqref{CW:macro:dyn} exhibits a phase transition at the critical value $\beta=1$. The solution $m=0$ is an equilibrium of \eqref{CW:macro:dyn} for all values of the parameters. For $\beta \leq 1$,  it is globally stable; whereas, for $\beta > 1$, it loses stability and two new stable fixed point $m = \pm m_\beta$, $m_\beta > 0$, bifurcate. We refer the reader to \cite{Ell85}. For later convenience, let us introduce the notation 
\[
G_{1,\beta}(x) = \cosh(\beta x) - x \sinh(\beta x) \; \mbox{ and } \; G_{2,\beta}(x) = \sinh(\beta x) - x \cosh(\beta x).
\]
Observe that the equilibria of \eqref{CW:macro:dyn} are solutions to $G_{2,\beta}(x)=0$.\\

 \subsection{Main results}

We want to discuss the moderate deviations behavior of the magnetization around its limiting stationary points in the various regimes. We have the following three results that can be obtained as particular cases of the more general Theorem \ref{theorem:mdp_1d_arbitrarypotential} stated and proven in Section~\ref{subsct:MDP&CLT:arbitrary:potential}. The first of our statements is mainly of interest for sub-critical inverse temperatures $\beta < 1$, but  is indeed valid for all $\beta \geq 0$. The results for the critical and super-critical regimes follow afterwards.

\begin{theorem}[Moderate deviations around $0$] \label{theorem:moderate_deviations_CW_subcritical}
Let $\{b_n\}_{n\geq 1}$ be a sequence of positive real numbers such that $b_n \to \infty$ and $b_n^{2} n^{-1} \to 0$. Suppose that $b_n m_n(0)$ satisfies the large deviation principle with speed $n b_n^{-2}$ on $\bR$ with rate function $I_0$. Then the trajectories $\left\{b_n m_n(t)\right\}_{t \geq 0}$ satisfy the large deviation principle on $D_\bR(\bR^+)$:
\begin{equation*}
\PR\left[\left\{b_n m_n(t)\right\}_{t \geq 0} \approx \{\gamma(t)\}_{t \geq 0}  \right] \sim e^{-n b_n^{-2} I(\gamma)},
\end{equation*}
where $I$ is the good rate function
\begin{equation}\label{CW:0MD:RF}
I(\gamma) = \begin{cases}
I_0(\gamma(0)) + \int_0^\infty  \mathcal{L} (\gamma(s),\dot{\gamma}(s)) \dd s & \text{if } \gamma \in \cA\cC, \\
\infty & \text{otherwise},
\end{cases}
\end{equation}
with
\[
\mathcal{L}(x,v) = \frac{1}{8} \left|v + 2x(1-\beta) \right|^2.
\]
\end{theorem}

\begin{theorem}[Moderate deviations: critical temperature $\beta = 1$] \label{theorem:moderate_deviations_CW_critical}
Let $\{b_n\}_{n\geq 1}$ be a sequence of positive real numbers such that $b_n \to \infty$ and $b_n^{4} n^{-1} \to 0$. Suppose that $b_n m_n(0)$ satisfies the large deviation principle with speed $n b_n^{-4}$ on $\bR$ with rate function $I_0$. Then the trajectories $\left\{b_n m_n(b_n^2 t)\right\}_{t \geq 0}$ satisfy the  large deviation principle on $D_\bR(\bR^+)$:
\begin{equation*}
\PR\left[\left\{b_n m_n(b_n^2 t)\right\}_{t \geq 0} \approx \{\gamma(t)\}_{t \geq 0}  \right] \sim e^{-n b_n^{-4} I(\gamma)},
\end{equation*}
where $I$ is the good rate function
\begin{equation}\label{CW:criticalMD:RF}
I(\gamma) = \begin{cases}
I_0(\gamma(0)) + \int_0^\infty \mathcal{L}(\gamma(s),\dot{\gamma}(s))  \dd s & \text{if } \gamma \in \cA\cC, \\
\infty & \text{otherwise},
\end{cases}
\end{equation}
with
\[
\mathcal{L}(x,v) = \frac{1}{8} \left|v + \frac{2}{3}x^3 \right|^2.
\]
\end{theorem}

\begin{theorem}[Moderate deviations: super-critical temperatures $\beta > 1$] \label{theorem:moderate_deviations_CW_supercritical}
Let \mbox{$m \in \{-m_\beta,+m_\beta\}$} be a non-zero solution of $G_{2,\beta}(x) = 0$. Moreover, let $\{b_n\}_{n\geq 1}$ be a sequence of positive real numbers such that $b_n \to \infty$ and $b_n^{2} n^{-1} \to 0$. Suppose that $b_n (m_n(0) - m)$ satisfies the large deviation principle with speed $n b_n^{-2}$ on $\bR$ with rate function $I_0$. Then the trajectories $\left\{b_n (m_n(t) - m)\right\}_{t \geq 0}$ satisfy the large deviation principle on $D_\bR(\bR^+)$:
\begin{equation*}
\PR\left[\left\{b_n (m_n(t) - m) \right\}_{t \geq 0} \approx \{\gamma(t)\}_{t \geq 0}  \right] \sim e^{-n b_n^{-2} I(\gamma)},
\end{equation*}
where $I$ is the good rate function
\begin{equation}\label{CW:supercriticalMD:RF}
I(\gamma) = \begin{cases}
I_0(\gamma(0)) + \int_0^\infty  \mathcal{L} (\gamma(s),\dot{\gamma}(s)) \dd s & \text{if } \gamma \in \cA\cC, \\
\infty & \text{otherwise},
\end{cases}
\end{equation}
with
\[
\mathcal{L}(x,v) = \frac{(v-2xG_{2,\beta}'(m))^2}{8 G_{1,\beta}(m)}.
\]
\end{theorem}

The rate functions \eqref{CW:0MD:RF} and \eqref{CW:supercriticalMD:RF} have a similar structure. Indeed, whenever $\beta < 1$, $m = 0$ is the unique solution of $G_{2,\beta}(x) = 0$; moreover, it yields $G_{1,\beta}(0) = 1$ and $G_{2,\beta}'(0) = \beta - 1$. \\
By choosing the sequence $b_n = n^\alpha$, with $\alpha > 0$, we can rephrase Theorems~\ref{theorem:moderate_deviations_CW_subcritical}, \ref{theorem:moderate_deviations_CW_critical} and \ref{theorem:moderate_deviations_CW_supercritical} in terms of more familiar ``moderate'' scalings involving powers of the volume. We therefore get estimates for the probability of a typical trajectory on a scale that is between a law of large numbers and a central limit theorem. We give a schematic summary of these special results in Table~\ref{tab:CW:deviations} below. 

\smallskip

For any of the three cases above, we define the \textit{Hamiltonian} $H : \bR \times \bR \rightarrow \bR$ by taking the Legendre transform of $\cL$: $H(x,p) = \sup_v pv - \cL(x,v)$. It is well known that the rate function $S$ of the stationary measures of the Markov processes, also known as the quasi-potential, solves the equation $H(x,S'(x)) = 0$, cf. Theorem 5.4.3 in \cite{FW98}. We use this property to show that our results are consistent with the moderate deviation principles obtained for the stationary measures in \cite{EiLo04}. We give the Hamiltonian of the three cases above
\begin{enumerate}[(a)]
\item \textit{sub-critical temperatures}, Theorem \ref{theorem:moderate_deviations_CW_subcritical} $H(x,p) = 2x(\beta - 1)p + 2p^2$,
\item \textit{critical temperature}, Theorem \ref{theorem:moderate_deviations_CW_critical} $H(x,p) = - \frac{2}{3}x^3p + 2p^2$,
\item \textit{super-critical temperatures}, Theorem \ref{theorem:moderate_deviations_CW_supercritical} $H(x,p) = 2xG_2'(m)p + 2 G_1(m) p^2$.
\end{enumerate}
The stationary rate function $S$ in each of these three cases obtained in \cite[Theorem 1.18]{EiLo04} is given by
\begin{enumerate}[(a)]
\item \textit{sub-critical temperatures}, $S(x) = \frac{1}{2}(1-\beta)x^2$,
\item \textit{critical temperature}, $S(x) = \frac{1}{12}x^4$,
\item \textit{super-critical temperatures}, $S(x) = \frac{1}{2}cx^2$, where $c := (\phi''(\beta m))^{-1} - \beta$, and $m$ is a solution of $G_{2,\beta}(x) = 0$ and $\phi(x) = \log \left(\cosh(x)\right)$.
\end{enumerate}
For (a) and (b), it is clear that $H(x,S'(x)) = 0$ for all $x$. For (c), since $m = \tanh(\beta m)$, by inverse function theorem we obtain $\phi''(\beta m) = 1 - m^2$. Therefore, we have
\begin{multline*}
c = (\phi''(\beta m))^{-1} - \beta = \left(1 - m \tanh(\beta m)\right)^{-1} - \beta \\
= - \left[\frac{\beta G_{1,\beta}(m) - \cosh(\beta m)}{G_{1,\beta}(m)} \right] = - \frac{G_{2,\beta}'(m)}{G_{1,\beta}(m)},
\end{multline*}
which implies that also in this case $H(x,S'(x)) = 0$ for all $x$.

The next theorem complements the results in \cite[Proposition~2.2]{CoDaP12} for the subcritical regime and shows that also in the supercritical case the fluctuations around a ferromagnetic stationary point  converge to a diffusion process. As previously, the statement is a direct consequence of a more general central limit theorem given in Theorem~\ref{theorem:CLT_1d_arbitrarypotential} below.

\begin{theorem}[Central limit theorem: super-critical temperatures $\beta > 1$] \label{theorem:CLT_CW_supercritical}
Let \mbox{$m \in \{-m_\beta,+m_\beta\}$} be a non-zero solution of $G_{2,\beta}(x) = 0$. Suppose that $n^{1/2} (m_n(0) - m)$ converges in law to $\nu$. Then the process $n^{1/2}(m_n(t) - m)$ converges weakly in law on $D_\bR(\bR^+)$ to the unique solution of:
\begin{equation}\label{CLT_CW_supercritical}
\begin{cases}
\dd Y(t) = 2Y(t)G_{2,\beta}'(m) \dd t + 2\sqrt{G_{1,\beta}(m)} \, \dd W(t) \\
Y(0) \sim \nu,
\end{cases}
\end{equation}
where $W(t)$ is a standard Brownian motion on $\bR$.
\end{theorem}

We want to conclude the analysis by considering moderate deviations and non-standard central limit theorem for volume-dependent temperatures decreasing to the critical point. In the sequel let $\{ m_n^\beta(t) \}_{t \geq 0}$ denote the process evolving at temperature~$\beta$.

\begin{theorem}[Moderate deviations: critical temperature $\beta = 1$, temperature rescaling] \label{theorem:moderate_deviations_CW_critical_temperature_rescaling}
Let $\kappa \geq 0$ and let $\{b_n\}_{n\geq 1}$ be a sequence of positive real numbers such that $b_n \rightarrow \infty$ and $b_n^4 n^{-1}\rightarrow 0$. Suppose that $b_n m_n^{1 + \kappa b_n^{-2}}(0)$ satisfies the large deviation principle with speed $n b_n^{-4}$ on $\bR$ with rate function $I_0$. Then the trajectories $\left\{b_n m_n^{1 + \kappa b_n^{-2}}(b_n^{2} t) \right\}_{t \geq 0}$ satisfy the  large deviation principle on $D_\bR(\bR^+)$:
\begin{equation*}
\PR\left[\left\{b_n m_n^{1 + \kappa b_n^{-2}}(b_n^2t)\right\}_{t \geq 0} \approx \{\gamma(t)\}_{t \geq 0}  \right] \sim e^{-n b_n^{-4}I(\gamma)},
\end{equation*}
where $I$ is the good rate function
\begin{equation}\label{CW:criticalMD:temp_resc:RF}
I(\gamma) = \begin{cases}
I_0(\gamma(0)) + \int_0^\infty \mathcal{L}(\gamma(s),\dot{\gamma}(s))\dd s & \text{if } \gamma \in \cA\cC, \\
\infty & \text{otherwise},
\end{cases}
\end{equation}
with
\[
\mathcal{L}(x,v) = \frac{1}{8}  \left| v - 2\left( \kappa x -  \frac{1}{3}x^3 \right) \right|^2\,.
\]
\end{theorem}

Notice that in this borderline case the moderate deviations rate function \eqref{CW:criticalMD:temp_resc:RF} is a sort of mixture of the one at the criticality \eqref{CW:criticalMD:RF} and the Gaussian rate function~\eqref{CW:0MD:RF}.

\begin{theorem}[Critical fluctuations: critical temperature $\beta=1$, temperature rescaling]
\label{theorem:CLT_CW_critical_temperature_rescaling}
Let $\kappa \geq 0$ and suppose that $n^{1/4} m_n^{1 + \kappa n^{-1/2}}(0)$ converges in law to $\nu$.

Then the process $n^{1/4}m_n^{1 + \kappa n^{-1/2}}(n^{1/2}t) $ converges weakly in law on $D_\bR(\bR^+)$ to the unique solution of:
\begin{equation}\label{CLT_CW_critical_temperature_rescaling}
\begin{cases}
\dd Y(t) = 2 \left[ \kappa Y(t) - \frac{1}{3} Y(t)^3\right] \dd t + 2 \, \dd W(t) \\
Y(0) \sim \nu,
\end{cases}
\end{equation}
where $W(t)$ is a standard Brownian motion on $\bR$.
\end{theorem}

The proof of Theorem~\ref{theorem:CLT_CW_critical_temperature_rescaling} is a simple adaptation of the proofs of Theorems~\ref{theorem:CLT_1d_arbitrarypotential} and \ref{theorem:moderate_deviations_CW_critical_temperature_rescaling} and therefore is omitted.\\
The results in the present section, together with the large deviation principle in \cite[Theorem 1]{Kr16b} and the study of fluctuations at $\beta=1$ in \cite[Theorem 2.10]{CoDaP12}, give a complete picture of the behaviour of fluctuations for the Curie-Weiss model. Indeed all the possible scales are covered. We summarize our findings in Table~\ref{tab:CW:deviations}. For completeness, we give also the Hamiltonian of the large deviation principle for the dynamics of $m_n(t)$ around its limiting trajectory \eqref{CW:macro:dyn}: 
\begin{equation} \label{eqn:Ham_LDP}
\begin{aligned}
H(x,p) & = \frac{1-x}{2} e^{\beta x} \left[e^{2p} - 1\right] + \frac{1+x}{2} e^{-\beta x}\left[e^{-2p} - 1\right] \\
& = \left[\cosh(2p) -1 \right]G_{1,\beta}(x) + \sinh(2p) G_{2,\beta}(x).
\end{aligned}
\end{equation}
The displayed conclusions are drawn under the assumption that in each case either the initial condition satisfies a large deviation principle at the correct speed or the initial measure converges weakly.\\
%{\color{blue}\sout{To conclude, it is worth to mention that analogous moderate deviation principles for the Curie-Weiss model at equilibrium can be found in \cite{EiLo04}.}}

\begin{table}[h!]
\caption{Fluctuations for the empirical magnetization of the Curie-Weiss spin-flip dynamics}
\begin{center}
\begin{tabular}{|c|c|c|c|}
\hline
\rowcolor{Gray}
\parbox[c][1cm][c]{1.5cm}{\scshape \footnotesize \centering Scaling Exponent} & \parbox[c][1cm][c]{2.5cm}{\scshape \footnotesize \centering Temperature} & \parbox[c][1cm][c]{2.8cm}{\scshape \footnotesize \centering Rescaled Process} & \parbox[c][1cm][c]{4cm}{\scshape \footnotesize \centering Limiting Theorem}\\
\hline \hline
$\alpha = 0$ & all $\beta$ & $m_n(t)$ & \parbox[c][1.3cm][c]{4cm}{\footnotesize \centering LDP at speed $n$ with Hamiltonian as in \eqref{eqn:Ham_LDP} \\ (see \cite{Co89,Kr16b})}\\
\hline\hline
%\rowcolor{LightGray}
\multicolumn{4}{|c|}{{\bf \scriptsize \cellcolor{LightGray} NON-CRITICAL CASES}}\\
\hline\hline
\multirow{3}{*}{$\alpha \in \left( 0, \frac{1}{2} \right)$} & all $\beta$ & $n^\alpha m_n(t)$ & \parbox[c][1.5cm][c]{4cm}{\footnotesize \centering LDP at speed $n^{1-2\alpha}$ with rate function \eqref{CW:0MD:RF}}\\
 & $\beta > 1$ & $n^\alpha (m_n(t) \pm m_\beta)$ &  \parbox[c][1.5cm][c]{4cm}{\footnotesize \centering LDP at speed $n^{1-2\alpha}$ with rate function \eqref{CW:supercriticalMD:RF}} \\
 \hline
 \multirow{3}{*}{$\alpha = \frac{1}{2}$} & all $\beta$ & $n^{1/2} m_n(t)$ & \parbox[c][2.7cm][c]{4cm}{\footnotesize \centering CLT \\ weak convergence to the unique solution of {\scriptsize \[ dY(t) = 2(\beta-1) Y(t) \dd t + 2 \dd W(t) \] (see \cite{CoDaP12})}}\\
 & $\beta > 1$ & $n^{1/2} (m_n(t) \pm m_\beta)$ &  \parbox[c][1.5cm][c]{4cm}{\footnotesize \centering CLT \\ weak convergence to the unique solution of \eqref{CLT_CW_supercritical}} \\ 
\hline\hline
%\rowcolor{LightGray}
\multicolumn{4}{|c|}{{\bf \scriptsize \cellcolor{LightGray} CRITICAL CASES}}\\
\hline \hline
 \multirow{3}{*}{$\alpha \in \left( 0, \frac{1}{4} \right)$} & $\beta = 1$ & $ n^{\alpha} m_n \left( n^{2\alpha} t \right)$ & \parbox[c][1.3cm][c]{4cm}{\footnotesize \centering LDP at speed $n^{1-4\alpha}$ with rate function \eqref{CW:criticalMD:RF}}\\
 & \parbox{2.5cm}{\centering $\beta = 1 + \kappa n^{-2\alpha}$ \mbox{ \footnotesize (with $\kappa \geq 0$)}} & $n^\alpha m_n \left( n^{2\alpha}t \right)$ &  \parbox[c][1.3cm][c]{4cm}{\footnotesize \centering LDP at speed $n^{1-4\alpha}$ with rate function \eqref{CW:criticalMD:temp_resc:RF}} \\
 \hline
 \multirow{3}{*}{$\alpha = \frac{1}{4}$} & $\beta = 1$ & $ n^{1/4} m_n \left( n^{1/2} t \right)$ & \parbox[c][2.3cm][c]{4cm}{\footnotesize \centering weak convergence to the unique solution of {\scriptsize \[ dY(t) = -\frac{2}{3} Y(t)^3 \dd t + 2 \dd W(t) \] (see \cite{CoDaP12})}}\\
 & \parbox{2.5cm}{\centering $\beta = 1 + \kappa n^{-1/2}$ \mbox{ \footnotesize (with $\kappa \geq 0$)}} & $n^{1/4} m_n \left( n^{1/2}t \right)$ &  \parbox[c][1.3cm][c]{4cm}{\footnotesize \centering weak convergence to the unique solution of \eqref{CLT_CW_critical_temperature_rescaling}} \\
\hline
\end{tabular}
\end{center}
\label{tab:CW:deviations}
\end{table}%
\section{Proofs}\label{sct:proofs}
%===============================================================================

%%%%%%%%%
\subsection{Strategy of the proof} \label{section:strategy_of_proof}
%%%%%%%%%

We will analyze large/moderate deviations behaviour following the Feng-Kurtz approach to large deviations \cite{FK06}. This method is based on three observations:
\begin{enumerate}
\item If the processes are exponentially tight, it suffices to establish the large deviation principle for finite dimensional distributions.
\item The large deviation principle for finite dimensional distributions can be established by proving that the semigroup of log Laplace-transforms of the conditional probabilities converges to a limiting semigroup.
\item One can often rewrite the limiting semigroup as a variational semigroup, which allows to rewrite the rate-function on the Skorohod space in Lagrangian form.
\end{enumerate}

The strategy to prove a large deviation principle with speed $r(n)$ for a sequence of Markov processes $\{X_n\}_{n \geq 1}$, having generators $\{ A_n \}_{n \geq 1}$, formally works as follows:

\begin{enumerate}
\item 
\emph{Identification of a limiting Hamiltonian $H$.}
The semigroups of log-Laplace transforms of the conditional probabilities
\begin{equation*}
V_n(t)f(x) = \frac{1}{r(n)} \log \bE\left[e^{r(n)f(X_n(t))} \, \middle| X_n(0) = x \right]
\end{equation*}
formally have generators $H_nf = r(n)^{-1}e^{-r(n)f}A_n e^{r(n)f}$. Then one verifies that the sequence $\{H_n\}_{n \geq 1}$ converges to a limiting operator $H$; i.e. one shows that, for any  $f \in \mathcal{D}(H)$, there exists $f_n \in \mathcal{D}(H_n)$ such that $f_n \to f$ and $H_n f_n \to Hf$, as $n \to \infty$. 
\item \emph{Exponential tightness.} Provided one can verify the exponential compact containment condition, the convergence of the sequence $\{H_n\}_{n \geq 1}$ gives exponential tightness. 
\item 
\emph{Verification of a comparison principle.} The theory of viscosity solutions gives applicable conditions for proving that the limiting Hamiltonian generates a semigroup. If for all $\lambda > 0$ and bounded continuous functions $h$, the Hamilton-Jacobi equation $f - \lambda H f = h$ admits a unique solution, one can extend the generator $H$ so that the extension satisfies the conditions of Crandall-Liggett theorem and thus generates a semigroup $V(t)$. Additionally, it follows that the semigroups $V_n(t)$ converge to $V(t)$, giving the large deviation principle. Uniqueness of the solution of the Hamilton-Jacobi equation can be established via comparison principle for sub- and super-solutions.
\item 
\emph{Variational representation of the limiting semigroup.} By Legendre transforming  the limiting Hamiltonian $H$, one can define a ``Lagrangian'' which can be used to define a variational semigroup and a variational resolvent. It can be shown that the variational resolvent provides a solution of the Hamilton-Jacobi equation and therefore,  by uniqueness of the solution, identifies the resolvent of $H$. As a consequence, an approximation procedure yields that the variational semigroup and the limiting semigroup $V(t)$ agree.  A standard argument is then sufficient to give a Lagrangian form of the path-space rate function.
\end{enumerate}

We refer to Appendix~\ref{sct:app:LDPviaHJequation} for an overview of the derivation of a large deviation principle via solution of Hamilton-Jacobi equation.\\ 
We proceed with the verification of the model-specific conditions needed to apply the results from the appendix. The treatment is carried out in a more abstract way than strictly necessary to single out important arguments that might get lost if all the objects are explicit.  The appendix and the following definitions are written to prove a path-space large deviation principle on $D_E(\bR^+)$, the Skorohod space of paths taking values in the closed set $E \subseteq \mathbb{R}^d$ which is contained in the $\bR^d$-closure of its $\bR^d$-interior. Two types of functions are of importance to this purpose: good penalization and good containment functions.

\begin{definition}
We say that $\{\Psi_\alpha\}_{\alpha >0}$, with $\Psi_\alpha : E^2 \rightarrow \bR$, are \textit{good penalization functions} if  there are extensions of $\Psi_\alpha$ to an open neighbourhood of $E^2$ in $\bR^d \times \bR^d$ (also denoted by $\Psi_\alpha$) so that
\begin{enumerate}[($\Psi$a)]
\item For all $\alpha > 0$, we have $\Psi_\alpha \geq 0$ and $\Psi_\alpha(x,y) = 0$ if and only if $x = y$. Additionally, $\alpha \mapsto \Psi_\alpha$ is increasing and
\begin{equation*}
\lim_{\alpha \rightarrow \infty} \Psi_\alpha(x,y) = \begin{cases}
0 & \text{if } x = y \\
\infty & \text{if } x \neq y.
\end{cases}
\end{equation*}	
\item $\Psi_\alpha$ is twice continuously differentiable in both coordinates for all $\alpha > 0$,
\item $(\nabla \Psi_\alpha(\cdot,y))(x) = - (\nabla \Psi_\alpha(x,\cdot))(y)$ for all $\alpha > 0$.
\end{enumerate}
\end{definition}

\begin{definition}
We say that $\Upsilon : E \rightarrow \bR$ is a \textit{good containment function} (for $H$) if there is an extension of $\Upsilon$ to an open neighbourhood of $E$ in $\bR^d$ (also denoted by $\Upsilon$) so that
\begin{enumerate}[($\Upsilon$a)]
\item $\Upsilon \geq 0$ and there exists a point $x_0 \in E$ such that $\Upsilon(x_0) = 0$,
\item $\Upsilon$ is twice continuously differentiable, 
\item for every $c \geq 0$, the set $\{x \in E \, | \, \Upsilon(x) \leq c\}$ is compact,
\item we have $\sup_{z \in E} H(z,\nabla \Upsilon(z)) < \infty$.
\end{enumerate}
\end{definition}

In the rest of the paper, we will  denote by $C_c^2(E)$the set of functions that are constant outside some compact set in $E$ and twice continuously differentiable on a neighbourhood of $E$ in~$\bR^d$. \\

Let us denote by $E_n$, closed subset of $E \subseteq \mathbb{R}^d$, the set where the finite $n$ process takes values.  Our large or moderate deviation principles will all follow from the application of Theorem \ref{theorem:Abstract_LDP} after having checked the following conditions:
\begin{enumerate}[(a)]
\item For all $f \in C_c^2(E)$ and compact sets $K \subseteq E$, we have 
\begin{equation*}
\lim_{n \rightarrow \infty} \sup_{x \in K \cap E_n} \left|H_n f(x) - Hf(x) \right| = 0.
\end{equation*}
\item There exists a good containment function $\Upsilon$ for $H$.
\item For all $\lambda > 0$ and $h \in C_b(E)$, the comparison principle holds for $f - \lambda Hf = h$. 
%It will be proven using good containment functions and Proposition \ref{proposition:comparison_conditions_on_H}.
\end{enumerate}

Rigorous definitions of Hamilton-Jacobi equation, viscosity solutions and comparison principle will be given in Appendix \ref{sct:app:LDPviaHJequation}.\\

The limiting Hamiltonians we will encounter are all of quadratic type and the space $E$ will always equal $\bR$. The following two known results establish (b) and (c) for Hamiltonians of this type, and are given for completeness. We postpone the verification of condition (a) for our various cases to Subsections \ref{subsct:MDP&CLT:arbitrary:potential} and \ref{subsct:MDP:critical:temp_rescaling}.

\begin{definition}
Let $E \subseteq \mathbb{R}^d$ be a closed set. We say that a vector field $\mathbf{F} : E \rightarrow \bR^d$ is one-sided Lipschitz if there exists a constant $M \geq 0$ such that, for all $x,y \in E$, it holds
\begin{equation*}
\ip{x-y}{\mathbf{F}(x) - \mathbf{F}(y)} \leq  M|x-y|^2.
\end{equation*}
\end{definition}

\begin{lemma}\label{lemma:FW_one_sided_lipschitz_containment_function}
Let $\mathbf{F}: \mathbb{R}^d \rightarrow \bR^d$ be  one-sided Lipschitz, $A$ a positive-definite matrix and $c \geq 0$. Assume $\mathbf{F}(0) = 0$. Consider the Hamiltonian $H$ with domain $C_c^2(\bR^d)$ and of the form $Hf(x) = H(x,\nabla f(x))$, where
\begin{equation*}
H(x,p) = \ip{p}{\mathbf{F}(x)} + c \, \ip{Ap}{p}.
\end{equation*}
Then $\Upsilon(x) =  \log (1+ \frac{1}{2}|x|^2)$ is a good containment function for $H$.
\end{lemma}

\begin{proof}
The proof is based on a simplification of Example 4.23 in \cite{FK06}. Since the $i$-th component of the gradient of $\Upsilon$ is given by 
\begin{equation*}
(\nabla \Upsilon(x))_i = \frac{x_i}{1+\frac{1}{2}|x|^2},
\end{equation*}
we obtain
\[
H(x,\nabla \Upsilon(x)) =  \frac{\ip{x}{\mathbf{F}(x)}}{1+\frac{1}{2}|x|^2} + \frac{c \ip{Ax}{x}}{(1+\frac{1}{2}|x|^2)^2} .
\]

Notice that, by the one-sided Lipschitz property of $\mathbf{F}$, there exists a constant $M$ such that $\ip{x}{\mathbf{F}(x)} = \ip{x- 0}{\mathbf{F}(x) - \mathbf{F}(0)} \leq M |x|^2$. Moreover, by Cauchy-Schwarz inequality, we have $\ip{Ax}{x} \leq |Ax||x| \leq \|A\| |x|^2$, where $\|A\| := \sup_{|x|=1} \, |Ax|$. Therefore, we get the estimate
\[
H(x,\nabla \Upsilon(x)) \leq \frac{M |x|^2 }{1+\frac{1}{2}|x|^2} + \frac{c \|A\|  |x|^2}{\left( 1+\frac{1}{2}|x|^2 \right)^2} \leq 4 (M +c \|A\|),
\]
which gives $\sup_x H(x,\nabla \Upsilon(x)) < \infty$, implying that $\Upsilon$ is a good containment function.
\end{proof}

\begin{proposition} \label{proposition:FW_one_sided_lipschitz_comparison_principle}
Let $\mathbf{F}: \mathbb{R}^d \rightarrow \bR^d$ be  one-sided Lipschitz, $A$ a positive-definite matrix and $c \in \mathbb{R}$. Assume $\mathbf{F}(0) = 0$. Consider the Hamiltonian $H$ with domain $C_c^2(\bR^d)$ and of the form $Hf(x) = H(x,\nabla f(x))$, where
\begin{equation*}
H(x,p) = \ip{p}{\mathbf{F}(x)} + c \, \ip{Ap}{p}.
\end{equation*}
Then, for every $\lambda > 0$ and $h \in C_b(\bR^d)$, the comparison principle is satisfied for $f - \lambda H f = h$.
\end{proposition}

\begin{proof}
We apply Proposition \ref{proposition:comparison_conditions_on_H}. We have to check \eqref{condH:negative:liminf}. We use the good containment function introduced in Lemma \ref{lemma:FW_one_sided_lipschitz_containment_function} and the collection of good penalization functions $\{\Psi_\alpha \}_{\alpha>0}$, with $\Psi_\alpha(x,y) = \frac{\alpha}{2}|x-y|^2$. We fix $\varepsilon > 0$ and write $x_\alpha, y_\alpha$ instead of $x_{\alpha,\varepsilon},y_{\alpha,\varepsilon}$ to lighten the notation. As the term $\ip{Ap}{p}$ does not depend on $x$ or $y$, we have

\begin{align*}
&H(x_\alpha, \alpha(x_\alpha - y_\alpha)) - H(y_\alpha, \alpha(x_\alpha - y_\alpha)) \\
&\qquad \qquad  = \ip{\alpha(x_\alpha - y_\alpha)}{\mathbf{F}(x_\alpha)} -\ip{\alpha(x_\alpha - y_\alpha)}{\mathbf{F}(y_\alpha)} \leq  M \Psi_{\alpha}(x_\alpha,y_\alpha).
\end{align*}
By Lemma \ref{lemma:doubling_lemma}, we obtain $\lim_{\alpha \to \infty} \Psi_{\alpha}(x_\alpha,y_\alpha) = 0$ and the conclusion follows.

\end{proof}

\begin{remark}
Lemma~\ref{lemma:FW_one_sided_lipschitz_containment_function} and Proposition~\ref{proposition:FW_one_sided_lipschitz_comparison_principle} can be suitably adapted (by centering~$\Upsilon$) to deal with the case when $\mathbf{F}(x_s) = 0$ at some $x_s \neq 0$.
\end{remark}

%%%%%%%%%%%%
\subsection{Proof of Theorems~\ref{theorem:moderate_deviations_CW_subcritical}, \ref{theorem:moderate_deviations_CW_critical}, \ref{theorem:moderate_deviations_CW_supercritical} and \ref{theorem:CLT_CW_supercritical}}
\label{subsct:MDP&CLT:arbitrary:potential}
%%%%%%%%%%%%

We introduce a general version of dynamics \`a la Curie-Weiss where the evolution of the magnetization is driven by a sufficiently smooth arbitrary potential. We characterize moderate deviations and central limit theorem for this generalization getting then as corollaries the statements of Theorems \ref{theorem:moderate_deviations_CW_subcritical}--\ref{theorem:CLT_CW_supercritical}.\\

Let $U$ be a continuously differentiable potential and consider the dynamics such that the empirical magnetization $\{m_n(t)\}_{t \geq 0}$ is a Markov process on $E_n$, with generator
\begin{multline}\label{eqn:CWgenerator_arbitrarypotential}
\cA_nf(x) = \frac{n(1-x)}{2} \, e^{U'(x)} \left[f\left(x + \frac{2}{n}\right) - f(x)\right] \\
%L2
+ \frac{n(1+x)}{2} \, e^{-U'(x)} \left[f\left(x - \frac{2}{n}\right) - f(x)\right].
\end{multline}
The infinite volume dynamics corresponding to the Markov process \eqref{eqn:CWgenerator_arbitrarypotential} may be derived from the large deviation principle in \cite[Theorem 1]{Kr16b}. In particular, the stationary points for the limiting dynamics are the solutions of $G_2(x)=0$, where
\[
G_{2}(x) := \sinh(U'(x)) - x \cosh(U'(x)).
\]
For later convenience we define also $G_{1}(x) := \cosh(U'(x)) - x \sinh(U'(x))$. In this setting, we have a moderate deviation principle and a weak convergence result.

\begin{theorem}[Moderate deviations, arbitrary potential and stationary point] \label{theorem:mdp_1d_arbitrarypotential}
Let $m$ be a solution of $G_2(x) = 0$. Let $k \in \bN \cup \{0\}$ and suppose that $U'$ is $2k+1$ times continuously differentiable. Additionally, suppose that
\begin{enumerate}[(a)]
\item $G_2^{(l)}(m) = 0$ for $l \leq 2k$,
\item if $k > 0$, then $G_2^{(2k+1)}(m) \leq 0$.
\end{enumerate}
Let $\{b_n\}_{n\geq 1}$ be a sequence of positive real numbers such that
\begin{equation*}
b_n \rightarrow \infty, \qquad \frac{b_n^{2(k+1)}}{n} \rightarrow 0.
\end{equation*}

Suppose that $b_n (m_n(0) - m)$ satisfies the large deviation principle with speed $n b_n^{-2(k+1)}$ on $\bR$ with rate function $I_0$. Then the trajectories $\left\{b_n(m_n(b_n^{2k} t) - m)\right\}_{t \geq 0}$ satisfy the large deviation principle on $D_\bR(\bR^+)$:
\begin{equation*}
\PR\left[\left\{b_n(m_n(b_n^{2k} t) - m)\right\}_{t \geq 0} \approx \{\gamma(t)\}_{t \geq 0}  \right] \sim e^{-n b_n^{-2(k+1)}I(\gamma)},
\end{equation*}
where $I$ is the good rate function
\begin{equation*}
I(\gamma) = \begin{cases}
I_0(\gamma(0)) + \int_0^\infty \cL(\gamma(s),\dot{\gamma}(s)) \dd s & \text{if } \gamma \in \cA\cC, \\
\infty & \text{otherwise},
\end{cases}
\end{equation*}
and 
\begin{equation*}
\cL(x,v) =  \frac{\left(v-\frac{2x^{2k+1}}{(2k+1)!}G_{2}^{(2k+1)}(m) \right)^2}{8 G_{1}(m)}.
\end{equation*}
\end{theorem}

\begin{proof}

The generator $A_n$ of the process $b_n(m_n(b_n^{2k}t) - m)$ can be deduced from \eqref{eqn:CWgenerator_arbitrarypotential} and is given by
\begin{align*}
A_n f (x) &= b_n^{2k} n \frac{1-m - xb_n^{-1}}{2} e^{U'(m + xb_n^{-1})} \left[f\left(x + 2b_n n^{-1}\right) - f(x)\right] \\
&+ b_n^{2k} n \frac{1 + m +xb_n^{-1}}{2} e^{- U'(m + xb_n^{-1})} \left[ f\left(x - 2b_n n^{-1}\right) - f(x) \right].
\end{align*} 
Therefore the Hamiltonian 
\begin{equation*}
H_nf = b_n^{2(k+1)}n^{-1} e^{-n b_n^{-2(k+1)}f} A_n e^{n b_n^{-2(k+1)}f},
\end{equation*}
results in
\begin{multline*}
H_nf(x) = b_n^{4k+2} \frac{1-m - xb_n^{-1}}{2} e^{U'(m + xb_n^{-1})} \left[e^{n b_n^{-2(k+1)}\left(f\left(x + 2b_n n^{-1}\right) - f(x)\right)}-1\right] \\
+ b_n^{4k+2} \frac{1 + m +xb_n^{-1}}{2} e^{- U'(m + xb_n^{-1})} \left[e^{n b_n^{-2(k+1)}\left(f\left(x - 2b_n n^{-1}\right) - f(x)\right)}-1\right].
\end{multline*}

We now prove the convergence of the sequence $H_nf$ for $f \in C_c^2(\bR)$. To compensate for the $b_n^{4k+2}$ up front, we Taylor expand the exponential containing $f$ up terms of $O(b_n^{-4k-2})$: 
\begin{multline*}
\exp\left\{n b_n^{-2(k+1)}\left(f(x \pm 2b_n n^{-1}) -f(x)\right)\right\} -1 \\
= \pm 2 b_n^{-2k -1} f'(x) + 2 b_n^{-4k-2}(f'(x))^2 + o(b_n^{-4k-2}).
\end{multline*}
Observe that the above expansion holds since $b_n^{-2k}n^{-1} = o(b_n^{-4k-2})$ by hypothesis. Thus, combining the terms with $f'$ and the terms with $(f')^2$, we find that
\begin{equation*}
H_nf(x) = 2 b_n^{2k +1} G_{2}(m + x b_n^{-1}) f'(x) + 2 G_{1}(m + x b_n^{-1}) (f'(x))^2 + o(1).
\end{equation*}
Next, we Taylor expand $G_1,G_2$ around $m$. For $G_1$ it is clear that only the zero'th order term remains. For $G_2$, we use that the first $2k$ terms disappear. This gives
\begin{align*}
H_nf(x) & = \frac{2x^{2k+1}}{(2k+1)!} G_{2}^{(2k+1)}(m) f'(x) + 2 G_{1}(m) (f'(x))^2 + o(1),
\end{align*}
where the $o(1)$ is uniform on compact sets. Thus, for $f \in C_c^2(\bR)$, $H_nf$ converges uniformly to $Hf(x) = H(x,f'(x))$ where 
\begin{equation*}
H(x,p) = \frac{2x^{2k+1}}{(2k+1)!}G_{2}^{(2k+1)}(m) p + 2 G_{1}(m) p^2.
\end{equation*}
The large deviation result follows by Theorem \ref{theorem:Abstract_LDP}, Lemma \ref{lemma:FW_one_sided_lipschitz_containment_function} and Proposition \ref{proposition:FW_one_sided_lipschitz_comparison_principle}. Note that condition (b) in the statement guarantees that the vector field is one-sided Lipschitz. The Lagrangian is found by taking a Legendre transform of~$H$.
\end{proof}

\begin{theorem}[Central limit theorem, general potential, arbitrary stable stationary point] \label{theorem:CLT_1d_arbitrarypotential}
Let $m$ be a solution of $G_2(x) = 0$. Let $k \in \bN \cup \{0\}$ and suppose that $U'$ is $2k+1$ times continuously differentiable. Additionally, suppose that
\begin{enumerate}[(a)]
\item $G_2^{(l)}(m) = 0$ for $l \leq 2k$,
\item if $k > 0$, then $G_2^{(2k+1)}(m) \leq 0$.
\end{enumerate}
Suppose that $n^{\frac{1}{2k+2}} (m_n(0) - m)$ converges in law to $\nu$. Then the process 
\[
n^{\frac{1}{2k+2}} \left( m_n \left( n^{\frac{k}{k+1}}t \right) - m \right)
\]
converges weakly in law on $D_\bR(\bR^+)$ to the unique solution of:
\begin{equation}\label{limiting_diffusion_CLT_CW}
\begin{cases}
\dd Y(t) = \frac{2}{(2k+1)!} \, Y(t)^{2k+1} G_{2}^{(2k+1)}(m)  \, \dd t +  2\sqrt{G_{1}(m)} \, \dd W(t) \\
Y(0) \sim \nu,
\end{cases}
\end{equation}
where $W(t)$ is a standard Brownian motion on $\bR$.
\end{theorem}

The limiting diffusion \eqref{limiting_diffusion_CLT_CW} admits a unique invariant measure with density proportional to $\exp \left\{ - c y^{2k+2}/(2k+2)! \right\}$, with $c = 4 \vert G_2^{(2k+1)}(m) \vert$. Observe that it is precisely the limiting distribution prescribed by the analysis of equilibrium fluctuations performed in \cite{ElNe78b}.

\smallskip

The proof of the theorem given below is in the spirit of the proofs of the moderate deviation principles and based on a combination of proving the compact containment condition and the convergence of the generators.

We first prove the compact containment condition. Let 
\[
X_n(t) := n^{\frac{1}{2k+2}} \left( m_n \left( n^{\frac{k}{k+1}} t \right) - m \right)
\]
be the space-time rescaled fluctuation process. We introduce the family $\{ \tau_n^C \}_{n \geq 1}$ of stopping times, defined by
\[
\tau_n^C := \inf_{t \geq 0} \left\{ \left\vert X_n(t) \right\vert \geq C \right\}\,.
\]
We start by studying the asymptotic behavior of the sequence $\{ \tau_n^C \}_{n \geq 1}$.

\begin{lemma}\label{lmm:asymptotics_stopping_times}
For any $T \geq 0$ and $\varepsilon > 0$, there exist $n_\varepsilon \geq 1$ and $C_\varepsilon > 0$ such that
\[
\sup_{n \geq n_\varepsilon} \, \mathbb{P} \left( \tau_n^{C_\varepsilon} \leq T \right) \leq \varepsilon\,.
\]
\end{lemma}

\begin{proof}
Let $C$ be a strictly positive constant. First observe that 
\begin{align}
\mathbb{P} \left( \tau_n^C \leq T \right) &\leq  \mathbb{P} \left( \sup_{0 \leq t \leq T \wedge \tau_n^C} \left\vert X_n(t) \right\vert \geq C \right). \label{target_probability}
%L2
%&=  \mathbb{P} \left( \sup_{0 \leq t \leq T \wedge \tau_n^C} X_n^2(t) \geq C^2 \right). 
\end{align}
We will obtain bounds for \eqref{target_probability} and show that it can be made arbitrarily small whenever $n$ is large enough. The idea is to get the estimate by considering a martingale decomposition for $\tilde{f}(X_n)$, where $\tilde{f} \in C^3(\bR)$ has bounded partial derivatives and is such that $\tilde{f}(x) = \tilde{f}(-x)$ and $\lim_{x \rightarrow \infty} \tilde{f}(x) = \infty$. \\
For any $n \geq 1$ and $j \in \{-1,+1\}$, let $N_n(j, \dd t)$ be the Poisson process counting the number of flips of spins with value $j$ up to time $n^{\frac{k}{k+1}} t$. The intensity of $N_n(j, \dd t)$ is  $R_n(j,t) \dd t$ with
\[
R_n (j,t) = \frac{n^{\frac{2k+1}{2k+2}}}{2} \,\left[ 1 + j \left( m + x n^{-\frac{1}{2k+2}} \right) \right] \, e^{-j U' \left( m + x n^{-\frac{1}{2k+2}} \right)}.
\]
Moreover, we define
\begin{equation}\label{eqn:Poisson_process_minus_intensity}
 \widetilde{N}_n(j, \dd t) := N_n(j, \dd t) - R_n (j, t) \, \dd t.
\end{equation}
Let $\tilde{f}(x) =  \log \sqrt{1+x^2}$ and consider the semi-martingale decomposition
\[
\tilde{f}(X_n(t)) = \tilde{f}(X_n(0)) + \int_0^t A_n \tilde{f}(X_n(s)) \, \dd s + M^{(1)}_n(t),
\]
where $M^{(1)}_n$ is the local martingale given by
\[
M^{(1)}_n (t) := \int_0^t \sum_{j = \pm 1} \left( \overline{\nabla}_j \, \tilde{f}(X_n(s)) \right)^2 \widetilde{N}_n(j, \dd s)
\]
and
\[
\overline{\nabla}_j \tilde{f}(X_n(t)) := \log \sqrt{1+ \left( X_n (t) - 2 j n^{-\frac{2k+1}{2k+2}}\right)^2} - \log \sqrt{1+X_n^2 (t)}.
\]
We have
\begin{multline*}
\mathbb{P} \left( \sup_{0 \leq t \leq T \wedge \tau_n^C} \left\vert X_n(t) \right\vert \geq C \right) \leq \mathbb{P} \left( \sup_{0 \leq t \leq T \wedge \tau_n^C} \tilde{f}(X_n(t)) \geq \tilde{f}(C) \right)  \\
\leq \mathbb{P} \left( \sup_{0 \leq t \leq T \wedge \tau_n^C} \tilde{f}(X_n(0)) \geq \frac{\tilde{f}(C)}{3} \right) + \mathbb{P} \left( \sup_{0 \leq t \leq T \wedge \tau_n^C} A_n \tilde{f}(X_n(t)) \geq \frac{\tilde{f}(C)}{3T} \right) \\
+ \mathbb{P} \left( \sup_{0 \leq t \leq T \wedge \tau_n^C} M^{(1)}_n(t) \geq \frac{\tilde{f}(C)}{3} \right).
\end{multline*}
We estimate the three terms in the right hand-side of previous inequality. All the constants appearing in the bounds below are \emph{independent of $n$}.
\begin{itemize}
\item 
Convergence in law of the initial condition implies $\mathbb{P} \left( \tilde{f}(X_n(0)) \geq c_0 (\varepsilon) \right) \leq \frac{\varepsilon}{3}$ for a sufficiently large $c_0(\varepsilon)$ and for all $n$.
\item
Since we are stopping the process $X_n(t)$ whenever it leaves the compact set $[-C,C]$, we find that $\tilde{f}(X_n(t))$ is bounded on the set of interest. Therefore, we can apply \eqref{eqn:inf_gen_space-time_rescaled_fluct_CLT} proven below to obtain
\begin{equation}\label{eqn:generator:ftilde}
A_n \tilde{f} (x) = \frac{2 G_2^{(2k+1)}(m)}{(2k+1)!} \frac{x^{2k+2}}{1+x^2} + 4 G_1(m) \frac{1-x^2}{(1+x^2)^2} + o(1),
\end{equation}
that, for $t \leq \tau_n^C$, implies 
\begin{equation}\label{estimate_generator}
A_n \tilde{f}(X_n (t))  \leq  c_1,
\end{equation}
for a sufficiently large $c_1>0$, independent of $C$. Note that the first term in the right hand-side of \eqref{eqn:generator:ftilde} is bounded if $k=0$ and negative if $k > 0$.
\item
Since 
\begin{align}
\left\vert R_n (j,t) \right\vert &\leq n^{\frac{2k+1}{2k+2}} \, e^{-j U' \left( m + x n^{-\frac{1}{2k+2}} \right)} \nonumber\\
%L2
&\leq n^{\frac{2k+1}{2k+2}} \, \left( e^{U'(m)} + O \left( n^{-\frac{1}{2k+2}}\right) \right)\label{estimate_intensity}
\end{align}
and
\begin{equation}\label{estimate_gradient}
\left( \overline{\nabla}_j \, \tilde{f}(X_n(t)) \right)^2 = \left(\int_{X_n(t)}^{X_n (t) - 2 j n^{-\frac{2k+1}{2k+2}}} \frac{y}{1+y^2} \, \dd y\right)^2 \leq  c_2 \, n^{-\frac{4k+2}{2k+2}},
\end{equation}
we obtain 
\begin{align}\label{estimate_quadratic_variation}
\mathbb{E} \left[ \left( M_n^{(1)} \left(T \wedge \tau_n^C\right)\right)^2 \right] &= \mathbb{E} \left[ \int_0^{T \wedge \tau_n^C} \sum_{j = \pm 1} \left( \overline{\nabla}_j \, \tilde{f}(X_n(s)) \right)^2 R_n(j,s) \dd s \right] \leq c_3 \, T
%L2
%&\leq \left( c_1 \, e^{U'(m)} + c_2 \right) T,
\end{align}
for sufficiently large $n$ and suitable $c_2, c_3 > 0$, independent of $C$. By Doob's maximal inequality, we can conclude
\[
\mathbb{P} \left( \sup_{0 < t < T \wedge \tau_n^C} M_n^{(1)} (t) \geq \frac{\tilde{f}(C)}{3} \right) \leq \frac{9Tc_3}{\tilde{f}(C)^2} \,.
\]
\end{itemize}

Therefore, for any $\varepsilon > 0$ and for sufficiently large $n$, by choosing the constant $C_\varepsilon \geq \max \left\{ c_0(\varepsilon), \sqrt{3c_1}, \tilde{f}^{-1} \left(\frac{27 T c_3}{\varepsilon} \right) \right\}$, we obtain $\sup_{n \geq n_\varepsilon} \mathbb{P} \left( \tau_n^{C_\varepsilon} \leq T \right) \leq \varepsilon$ as wanted.
\end{proof}

\begin{proof}[Proof of Theorem \ref{theorem:CLT_1d_arbitrarypotential}]

As introduced above, let $X_n(t) := n^{\frac{1}{2k+2}} \left( m_n \left( n^{\frac{k}{k+1}} t \right) - m \right)$ be the space-time rescaled process. The infinitesimal generator of the process  $X_n(t)$ can be easily deduced from \eqref{eqn:CWgenerator_arbitrarypotential}. It yields 
\begin{multline*}
A_n f(x) = n^{\frac{2k+1}{k+1}} \, \frac{1 - m - xn^{-\frac{1}{2k+2}}}{2} \, e^{U' \big( m + xn^{-\frac{1}{2k+2}} \big)} \left[f \left( x + 2n^{-\frac{2k+1}{2k+2}} \right) -f(x) \right] \\
%L2
+ n^{\frac{2k+1}{k+1}} \, \frac{1 + m + xn^{-\frac{1}{2k+2}}}{2} \, e^{-U' \big( m + xn^{-\frac{1}{2k+2}} \big)}\left[ f \left( x - 2n^{-\frac{2k+1}{2k+2}} \right)  -f(x) \right].
\end{multline*}
We want to characterize the limit of the sequence $A_n f$ for $f \in C^3_c(\mathbb{R})$, the set of three times continuously differentiable functions that are constant outside of a compact set. We first Taylor expand $f$ up to the second order
\begin{equation*}
f \left( x \pm 2n^{-\frac{2k+1}{2k+2}} \right) -f(x) = \pm 2n^{-\frac{2k+1}{2k+2}} f'(x) + 2n^{-\frac{2k+1}{k+1}} f''(x) + o \left( n^{-\frac{2k+1}{k+1}} \right)
\end{equation*}
and then, by combining the terms with $f'$ and the terms with $f''$, we obtain
\begin{multline*}
A_n f(x) = 2 n^{\frac{2k+1}{2k+2}} \, G_2 \left( m + x n^{-\frac{1}{2k+2}} \right) f'(x) \\
%L2
+ 2 \, G_1 \left( m + x n^{-\frac{1}{2k+2}} \right) f''(x) + o (1).
\end{multline*}
Now we Taylor expand $G_1$ and $G_2$ around $m$. For $G_2$ we use that the first $2k$ terms in the expansion vanish. As far as $G_1$, only the zero-th order term matters. Therefore it yields
\begin{equation}\label{eqn:inf_gen_space-time_rescaled_fluct_CLT}
A_n f(x) = \frac{2}{(2k+1)!} G_2^{(2k+1)}(m) \, x^{2k+1} f'(x) + 2 \, G_1( m) f''(x) + o(1).
\end{equation}
In other words, we conclude that for $f \in C_c^3(\bR)$ we have $\lim_n \vn{A_nf - Af} = 0$.

\smallskip

To prove the weak convergence result, we verify the conditions for Corollary 4.8.16 in \cite{EK86}. The martingale problem for the operator $(A,C_c^3(\bR))$ has a unique solution by Theorem 8.2.6 in \cite{EK86}. Additionally, the set $C_c^3(\bR)$ is an algebra that separates points. Finally, by Lemma \ref{lmm:asymptotics_stopping_times} the sequence $\{X_n\}_{n \geq 1}$ satisfies the compact containment condition. Thus the result follows by an application of Corollary 4.8.16 in \cite{EK86}. 

\end{proof}

The results in Section~\ref{sct:results} are recovered by setting $U(x) = \frac{\beta x^2}{2}$, with $\beta > 0$. Note that simply by choosing $U(x) = \frac{\beta x^2}{2} + B x$, with $\beta, B > 0$, we get the corresponding results for the Curie-Weiss with magnetic field.

%%%%%%%%%%%%
\subsection{Proof of Theorem \ref{theorem:moderate_deviations_CW_critical_temperature_rescaling}} \label{subsct:MDP:critical:temp_rescaling}
%%%%%%%%%%%%

The infinitesimal generator $A_n$ of the process $b_n m_n^{1 + \kappa b_n^{-2}}(b_n^2t)$ can be
easily deduced from \eqref{eqn:CWgenerator_arbitrarypotential} by using $U(x) = \frac{(1 + \kappa b_n^{-2})x^2}{2}$. The Hamiltonian 
\begin{equation*}
H_nf = b_n^{4}n^{-1} e^{-n b_n^{-4}f} A_n e^{n b_n^{-4}f},
\end{equation*}
is given by
\begin{multline*}
H_nf(x) = b_n^{6} \frac{1 - xb_n^{-1}}{2} \,e^{(1+\kappa b_n^{-2})xb_n^{-1}} \left[e^{n b_n^{-4}\left(f\left(x + 2b_n n^{-1}\right) - f(x)\right)}-1\right] \\
+ b_n^{6} \frac{1 +xb_n^{-1}}{2} \, e^{- (1+\kappa b_n^{-2})xb_n^{-1}} \left[e^{n b_n^{-4}\left(f\left(x - 2b_n n^{-1}\right) - f(x)\right)}-1\right].
\end{multline*}
We start by studying the limiting behaviour of the sequence $H_nf$ for $f \in C_c^2(\bR)$. Let $A_n$ be the generator of the process that has been speeded up by a factor $b_n^{2}$, i.e. $A_n f = b_n^{2} \cA_n f$. To compensate for the $b_n^6$ up front, we Taylor expand the exponential containing $f$ up terms of $O(b_n^{-6})$: 
\[
\exp\left\{n b_n^{-4}\left(f(x \pm 2b_n n^{-1}) -f(x)\right)\right\} -1 = \pm 2 b_n^{-3} f'(x) + 2 b_n^{-6}(f'(x))^2 + o(b_n^{-6}).
\]
Thus, combining the terms with $f'$ and the terms with $(f')^2$, we find that
\begin{multline*}
H_nf(x) = 2 \left[ b_n^3 \sinh \left( x b_n^{-1} + x \kappa b_n^{-3} \right) - x b_n^2 \cosh \left( x b_n^{-1} + x \kappa b_n^{-3} \right) \right]f'(x) \\
+ 2 (f'(x))^2 + o(1).
\end{multline*}
By Taylor expanding the hyperbolic functions
\[
\sinh \left( x b_n^{-1} + x \kappa b_n^{-3} \right) = x b_n^{-1} + x \kappa b_n^{-3} + \frac{1}{6}(x b_n^{-1} + x \kappa b_n^{-3})^3 + o \left( b_n^{-3} \right)
\]
\[
\cosh \left( x b_n^{-1} + x \kappa b_n^{-3} \right) = 1 + \frac{1}{2}(x b_n^{-1} + x\kappa b_n^{-3})^2 + o \left( b_n^{-2} \right),
\]
we get
\begin{equation*}
H_nf(x) = 2 \left[ \kappa x - \frac{1}{3} x^3 \right] f'(x) + 2 (f'(x))^2 + o(1),
\end{equation*}
where the $o(1)$ is uniform on compact sets. Thus, for $f \in C_c^2(\bR)$, $H_nf$ converges uniformly to $Hf(x) = H(x,f'(x))$ where 
\begin{equation*}
H(x,p) = 2 \left[ \kappa x - \frac{1}{3} x^3 \right] p + 2 p^2
\end{equation*}
The large deviation result follows by Theorem~\ref{theorem:Abstract_LDP}, Lemma~\ref{lemma:FW_one_sided_lipschitz_containment_function} and Proposition~\ref{proposition:FW_one_sided_lipschitz_comparison_principle}. The Lagrangian is found by taking a Legendre transform of $H$.

\appendix 

%===============================================================================
\section{Appendix: Large deviation principle via the Hamilton-Jacobi equation}
\label{sct:app:LDPviaHJequation}
%===============================================================================

In the Appendix, we will explain the basic steps to prove the path-space large deviation principle via uniqueness of solutions to the Hamilton-Jacobi equation. These steps follow the proofs in \cite{CIL92,FK06,DFL11} and have also been used in the proof of the large deviation principle for the dynamics of variants of the Curie-Weiss model via well-posedness of the Hamilton-Jacobi equation in \cite{Kr16b}.

First, we prove an abstract result on how to obtain uniqueness of viscosity solution of the Hamilton-Jacobi equation via the comparison principle. Then, we will state a result on how uniqueness, together with a exponential compact containment condition, yields the large deviation principle. The verification of the conditions for this result have already been carried out in Section \ref{section:strategy_of_proof}. 

We make the remark that the requirements on our space $E$ in Section \ref{section:appendix_ldp} and are more stringent than the ones in Sections \ref{section:appendix_definitions} and \ref{section:abstract_proof_of_comparison_principle}. The definitions of good penalization functions and of a good containment function are unchanged for the next two sections.

\subsection{Viscosity solutions for the Hamilton-Jacobi equation} \label{section:appendix_definitions}

Fix some $d \geq 1$. In this section, and in Section \ref{section:abstract_proof_of_comparison_principle}, let $E$ be a subset of $\bR^d$ that is contained in the $\bR^d$-closure of its $\bR^d$-interior. Additionally, assume that $E$ is a Polish space when equipped with its subspace-topology. 

\begin{remark}
The assumption that $E$ is contained in the $\bR^d$-closure of its $\bR^d$-interior is needed to ensure that the gradient $\nabla f(x)$ of a function $f \in C_b^1(\bR^d)$ for $x \in E$ is determined by its values in $E$.
\end{remark}

\begin{remark}
We say that $A \subseteq \bR^d$ is a $G_\delta$ set if it is the countable intersection of open sets in $\bR^d$. By Theorems 4.3.23 and 4.3.24 in \cite{En89}, we find that $E$ is Polish if and only if it is a $G_\delta$ set in $\bR^d$.
\end{remark}

Let $H : E \times \bR^d \rightarrow \bR$ be a continuous map. For $\lambda > 0$ and $h \in C_b(E)$, we will solve the \textit{Hamilton-Jacobi} equation
\begin{equation} \label{eqn:differential_equation_intro}
f(x) - \lambda H(x, \nabla f(x)) = h(x)  \qquad x \in E,
\end{equation}
in the \textit{viscosity} sense.

\begin{definition} \label{definition:viscosity} 
We say that $u$ is a \textit{(viscosity) subsolution} of equation \eqref{eqn:differential_equation_intro} if $u$ is bounded, upper semi-continuous and if, for every $f \in \cD(H)$ such that $\sup_x u(x) - f(x) < \infty$ and every sequence $x_n \in E$ such that
\begin{equation*}
\lim_{n \rightarrow \infty} u(x_n) - f(x_n)  = \sup_x u(x) - f(x),
\end{equation*}
we have
\begin{equation*}
\lim_{n \rightarrow \infty} u(x_n) - \lambda Hf(x_n) - h(x_n) \leq 0.
\end{equation*}
We say that $v$ is a \textit{(viscosity) supersolution} of equation \eqref{eqn:differential_equation_intro} if $v$ is bounded, lower semi-continuous and if, for every $f \in \cD(H)$ such that $\inf_x v(x) - f(x) > - \infty$ and every sequence $x_n \in E$ such that
\begin{equation*}
\lim_{n \rightarrow \infty} v(x_n) - f(x_n)  = \inf_x v(x) - f(x),
\end{equation*}
we have
\begin{equation*}
\lim_{n \rightarrow \infty} v(x_n) - \lambda Hf(x_n) - h(x_n) \geq 0.
\end{equation*}
\end{definition}

At various points, we will refer to \cite{FK06}. The notion of viscosity solution used here corresponds to the notion of \textit{strong viscosity solution} in \cite{FK06}. For operators of the form $Hf(x) = H(x,\nabla f(x))$ these two notions are equivalent. See also Lemma 9.9 in \cite{FK06}.

\begin{definition} 
We say that equation \eqref{eqn:differential_equation_intro} satisfies the \textit{comparison principle} if for a subsolution $u$ and supersolution $v$ we have $u \leq v$.
\end{definition}

Note that if the comparison principle is satisfied, then a viscosity solution is unique. 

To prove the comparison principle, we extend our scope by considering viscosity sub- and supersolutions to the Hamilton-Jacobi equation with two different operators that extend the original Hamiltonian in a suitable way.

Let $M(E,\overline{\bR})$ be the set of measurable functions from $E$ to $\overline{\bR} := \bR \cup \{ \infty \}$.

\begin{definition}
We say that $H_\dagger \subseteq M(E,\overline{\bR}) \times M(E,\overline{\bR})$ is a \textit{viscosity sub-extension} of $H$ if $H \subseteq H_\dagger$ and if for every $\lambda >0$ and $h \in C_b(E)$ a viscosity subsolution to $f-\lambda H f = h$ is also a viscosity subsolution to $f - \lambda H_\dagger f = h$. Similarly, we define a \textit{viscosity super-extension}.
\end{definition}

\begin{definition}
Consider two operators $H_\dagger, H_\ddagger  \subseteq M(E,\overline{\bR}) \times M(E,\overline{\bR})$ and pick $h \in C_b(E)$ and $\lambda > 0$. We say that the equations
\begin{equation*}
f - \lambda H_\dagger f = h, \qquad f - \lambda H_\ddagger f = h
\end{equation*}
satisfy the comparison principle if any subsolution $u$ to the first and any supersolution $v$ to the second equation satisfy $u \leq v$.
\end{definition}

We have the following straightforward result.

\begin{lemma}
Suppose that $H_\dagger$ and $H_{\ddagger}$ are a sub- and superextension of $H$ respectively. Fix $\lambda > 0$ and $h \in C_b(E)$. If the comparison principle is satisfied for the equations
\begin{equation*}
f - \lambda H_\dagger f = h, \qquad f - \lambda H_\ddagger f = h,
\end{equation*}
then the comparison principle is satisfied for
\begin{equation*}
f - \lambda H f = h.
\end{equation*}
\end{lemma}

\subsection{Abstract proof of the comparison principle} \label{section:abstract_proof_of_comparison_principle}

We introduce two convenient viscosity extensions of a particular Hamiltonian $H$ in terms of good penalization functions $\{\Psi_\alpha\}_{\alpha \geq 0}$ and containment function $\Upsilon$.
\begin{align*}
\cD(H_\dagger) & := C^1_b(E) \cup \left\{x \mapsto  (1-\varepsilon)\Psi_\alpha(x,y) + \varepsilon \Upsilon(x) +c \, \middle| \, \alpha,\varepsilon > 0, c \in \bR \right\}, \\
\cD(H_\ddagger) & := C^1_b(E) \cup \left\{y \mapsto - (1+\varepsilon)\Psi_\alpha(x,y) - \varepsilon \Upsilon(y) +c \, \middle| \, \alpha,\varepsilon > 0, c \in \bR \right\}.
\end{align*}
For $f \in \cD(H_\dagger)$, set $H_\dagger f(x) = H(x,\nabla f(x))$ and for $f \in \cD(H_\ddagger)$, set $H_\ddagger f(x) = H(x,\nabla f(x))$.

\begin{lemma} \label{lemma:viscosity_extension}
The operator $(H_\dagger,\cD(H_\dagger))$ is a viscosity sub-extension of $H$ and $(H_\ddagger,\cD(H_\ddagger))$ is a viscosity super-extension of $H$.
\end{lemma}

In the proof we need Lemma 7.7 from \cite{FK06}. We recall it here for the sake of readability. Let $M_\infty(E,\overline{\bR})$ denote the set of measurable functions $f : E \rightarrow \bR \cup \{\infty\}$ that are bounded from below.
\begin{lemma}[Lemma 7.7 in \cite{FK06}] \label{lemma:extension_lemma_7.7inFK}
Let $H$ and $H_\dagger \subseteq M_\infty(E,\overline{\bR}) \times M(E,\overline{\bR})$ be two operators. Suppose that for all $(f,g) \in H_\dagger$ there exist $\{(f_n,g_n)\} \subseteq H_\dagger$ that satisfy the following conditions:
\begin{enumerate}[(a)]
\item For all $n$, the function $f_n$ is lower semi-continuous.
\item For all $n$, we have $f_n \leq f_{n+1}$ and $f_n \rightarrow f$ point-wise.
\item Suppose $x_n \in E$ is a sequence such that $\sup_n f_n(x_n) < \infty$ and $\inf_n g_n(x_n) > - \infty$, then $\{x_n\}_{n \geq 1}$ is relatively compact and if a subsequence $x_{n(k)}$ converges to $x \in E$, then
\begin{equation*}
\limsup_{k \rightarrow \infty} g_{n(k)}(x_{n(k)}) \leq g(x).
\end{equation*}
\end{enumerate}
Then $H_\dagger$ is a viscosity sub-extension of $H$.\\
An analogous result holds for super-extensions $H_{\ddagger}$ by taking $f_n$ a decreasing sequence of upper semi-continuous functions and by replacing requirement (c) with
\begin{enumerate}
\item[(c$^{\prime}$)] Suppose $x_n \in E$ is a sequence such that $\inf_n f_n(x_n) > - \infty$ and $\sup_n g_n(x_n) <  \infty$, then $\{x_n\}_{n \geq 1}$ is relatively compact and if a subsequence $x_{n(k)}$ converges to $x \in E$, then
\begin{equation*}
\liminf_{k \rightarrow \infty} g_{n(k)}(x_{n(k)}) \geq g(x).
\end{equation*}
\end{enumerate} 
\end{lemma}

\begin{proof}[Proof of Lemma \ref{lemma:viscosity_extension}]
We only prove the sub-extension part. 

Consider a collection of smooth functions $\phi_n : \bR \rightarrow \bR$ defined as $\phi_n(x) = x$ if $x \leq n$ and $\phi_n(x) = n+1$ for $x \geq n+1$. Note that $\phi_{n + 1} \geq \phi_n$ for all $n$.

\smallskip

Fix a function $f \in \cD(H_\dagger)$ of the type $f(x) = (1-\varepsilon)\Psi_\alpha(x,y)+\varepsilon \Upsilon(x) + c$ and write $g = H_\dagger f$.

Because $\Psi_\alpha$ and $\Upsilon$ are good penalization and good containment functions, $f$ has a twice continuously differentiable extensions to a neighbourhood of $E$ in $\bR^d$. We will denote this extension also by $f$. Set $f_n = \phi_n \circ f$. Since $f$ is bounded from below and $\Upsilon$ has compact level sets, we find that $f_n$ is constant outside a compact set in $E$. Furthermore, $f$ and $\phi_n$ are twice continuously differentiable, so that we find $f_n \in C_c^2(E)$. We obtain $f_n \in \cD(H)$ and write $g_n = H f_n$.

We verify conditions (a)-(c) of Lemma \ref{lemma:extension_lemma_7.7inFK}. (a) has been verified above. (b) is a consequence of the fact that $n \mapsto \phi_n$ is increasing. For (c), let $\{x_n\}_{n \geq 1}$ be a sequence such that $\sup_n f_n(x_n) = M < \infty$. It follows by the compactness of the level sets of $\Upsilon$ and the positivity of $\Psi_\alpha$ that the set
\begin{equation*}
K := \{z \in E \, | \, f(z) \leq M\}
\end{equation*}
is compact. Thus the sequence $\{x_n\}$ is relatively compact, and in particular, there exist converging subsequences $x_{n(k)}$ with limits $x \in K$. For any such subsequence, we show that $\limsup_k g_{n(k)} (x_{n(k)}) \leq g(x)$.

\smallskip

As $\Psi_\alpha$ and $\Upsilon$, and thus $f$, are twice continuously differentiable up to a neighbourhood $U$ of $E$ in $\bR^d$, we find that the set
\begin{equation*}
V := \{z \in U \, | \, f(z) < M+1\}
\end{equation*}
is open in $\bR^d$ and contains $K$. For two arbitrary continuously differentiable functions $h_1,h_2$ on $U$, if $h_1(z) = h_2(z)$ for all $z \in V$, then $\nabla h_1(z) = \nabla h_2(z)$ for all $z \in V$.

\smallskip

Now suppose $x_{n(k)}$ is a subsequence in $K$ converging to some point $x \in K$. As $f$ is bounded on $V$, there exists a sufficiently large $N$ such that for all $n \geq N$ and $y \in V$, we have $f_n(y) = f(y)$. We conclude $\nabla f_n(y) = \nabla f(y)$ for $y \in K \subseteq V$ and hence
\begin{equation*}
g_n(y) = H(y,\nabla f_n(y)) = H(y,\nabla f(y)) = g(y).
\end{equation*}
In particular, we find $\limsup_{k} g_{n(k)}(x_{n(k)}) \leq g(x)$.
\end{proof}

We have the following variants of Lemma 9.2 in \cite{FK06} and Proposition 3.7 in \cite{CIL92}. Note that the presence of the containment function $\Upsilon$ makes sure that the suprema are attained. This motivates the name containment function: $\Upsilon$ forces the maxima to be in some compact set. 

\begin{lemma}\label{lemma:doubling_lemma}
Let $u$ be bounded and upper semi-continuous, let $v$ be bounded and lower semi-continuous, let $\Psi_\alpha : E^2 \rightarrow \bR^+$ be good penalization functions and let $\Upsilon$ be a good containment function.

\smallskip

Fix $\varepsilon > 0$. For every $\alpha >0$ there exist points $x_{\alpha,\varepsilon},y_{\alpha,\varepsilon} \in E$, such that
\begin{multline*}
\frac{u(x_{\alpha,\varepsilon})}{1-\varepsilon} - \frac{v(y_{\alpha,\varepsilon})}{1+\varepsilon} - \Psi_\alpha(x_{\alpha,\varepsilon},y_{\alpha,\varepsilon}) - \frac{\varepsilon}{1-\varepsilon}\Upsilon(x_{\alpha,\varepsilon}) -\frac{\varepsilon}{1+\varepsilon}\Upsilon(y_{\alpha,\varepsilon}) \\
= \sup_{x,y \in E} \left\{\frac{u(x)}{1-\varepsilon} - \frac{v(y)}{1+\varepsilon} -  \Psi_\alpha(x,y)  - \frac{\varepsilon}{1-\varepsilon}\Upsilon(x) - \frac{\varepsilon}{1+\varepsilon}\Upsilon(y)\right\}.
\end{multline*}
Additionally, for every $\varepsilon > 0$ we have that
\begin{enumerate}[(a)]
\item The set $\{x_{\alpha,\varepsilon}, y_{\alpha,\varepsilon} \, | \,  \alpha > 0\}$ is relatively compact in $E$.
\item All limit points of $\{(x_{\alpha,\varepsilon},y_{\alpha,\varepsilon})\}_{\alpha > 0}$ are of the form $(z,z)$ and for these limit points we have $u(z) - v(z) = \sup_{x \in E} \left\{u(x) - v(x) \right\}$.
\item Suppose $\Psi_\alpha$ can be written as $\Psi_\alpha = \alpha \Psi$, where $\Psi \geq 0$. Then we have 
\[
\lim_{\alpha \rightarrow \infty}  \alpha \Psi(x_{\alpha,\varepsilon},y_{\alpha,\varepsilon}) = 0.
\]
\end{enumerate}
\end{lemma}

\begin{proof}
The proof essentially follows the one of Proposition 3.7 in \cite{CIL92}. \\
Fix $\varepsilon > 0$. As $\Upsilon$ is a good containment function, its level sets are compact. This property, combined with the boundedness of $u$ and $v$ and the non-negativity of $\Psi_\alpha$, implies that the supremum can be restricted to a compact set $K_\varepsilon \subseteq E \times E$ that is independent of $\alpha > 0$.  
As $u$ is upper semi-continuous, and $v$, $\Psi_\alpha$ and $\Upsilon$ are lower semi-continuous, the supremum is attained for some pair $(x_{\alpha,\varepsilon},y_{\alpha,\varepsilon}) \in K_\varepsilon$. This proves (a).

\smallskip

We proceed with the proof of (b). Let $(x_0,y_0)$ be a limit point of $\{(x_{\alpha,\varepsilon},y_{\alpha,\varepsilon})\}_{\alpha > 0}$ such that $x_0 \neq y_0$. Without loss of generality, assume that $(x_{\alpha,\varepsilon},y_{\alpha,\varepsilon}) \rightarrow (x_0,y_0)$. By property $(\Psi a)$, the map $\alpha \mapsto \Psi_\alpha$ is increasing. Thus, for all $\alpha_0$ we have that
\begin{equation*}
\liminf_{\alpha \rightarrow \infty} \Psi_\alpha(x_{\alpha,\varepsilon},y_{\alpha,\varepsilon}) \geq \liminf_{\alpha \rightarrow \infty} \Psi_{\alpha_0}(x_{\alpha,\varepsilon},y_{\alpha,\varepsilon}) \geq \Psi_{\alpha_0}(x_0,y_0)
\end{equation*}
by the lower semi-continuity of $\Psi_{\alpha_0}$. Thus, we conclude that 
\begin{equation*}
\liminf_{\alpha \rightarrow \infty} \Psi_\alpha(x_{\alpha,\varepsilon},y_{\alpha,\varepsilon}) = \infty
\end{equation*}
as $\lim_{\alpha \rightarrow \infty} \Psi_\alpha(x,y) = \infty$ for all $x \neq y$. This contradicts the boundedness of $u$ and $v$.

\smallskip

We now prove (c). Let us define the constants
\begin{align*}
M_\alpha & := \frac{u(x_{\alpha,\varepsilon})}{1-\varepsilon} - \frac{v(y_{\alpha,\varepsilon})}{1+\varepsilon} - \Psi_\alpha(x_{\alpha,\varepsilon},y_{\alpha,\varepsilon}) - \frac{\varepsilon}{1-\varepsilon}\Upsilon(x_{\alpha,\varepsilon}) -\frac{\varepsilon}{1+\varepsilon}\Upsilon(y_{\alpha,\varepsilon}) \\
& = \sup_{x,y \in E} \left\{\frac{u(x)}{1-\varepsilon} - \frac{v(y)}{1+\varepsilon} -  \Psi_\alpha(x,y)  - \frac{\varepsilon}{1-\varepsilon}\Upsilon(x) - \frac{\varepsilon}{1+\varepsilon}\Upsilon(y)\right\}.
\end{align*}
Observe that the sequence $M_\alpha$ is decreasing as $\alpha \mapsto \Psi_\alpha$ is increasing point-wise. Moreover, the limit $\lim_{\alpha \rightarrow \infty} M_\alpha$ exists, since functions $u$ and $v$ are bounded from below. 
For any $\alpha > 0$, we obtain
\begin{align*}
M_{\alpha/2} & \geq \frac{u(x_{\alpha,\varepsilon})}{1-\varepsilon} - \frac{v(y_{\alpha,\varepsilon})}{1+\varepsilon} - \Psi_{\alpha/2}(x_{\alpha,\varepsilon},y_{\alpha,\varepsilon}) - \frac{\varepsilon}{1-\varepsilon}\Upsilon(x_{\alpha,\varepsilon}) -\frac{\varepsilon}{1+\varepsilon}\Upsilon(y_{\alpha,\varepsilon}) \\
& \geq  M_\alpha + \Psi_\alpha \left( x_{\alpha,\varepsilon}, y_{\alpha, \varepsilon} \right) - \Psi_{\alpha/2} \left( x_{\alpha,\varepsilon}, y_{\alpha, \varepsilon} \right)\\
%L2
& \geq M_\alpha  +\frac{\alpha}{2} \, \Psi \left( x_{\alpha,\varepsilon}, y_{\alpha, \varepsilon} \right)\\
%L3
& \geq M_\alpha,
\end{align*}
that implies $\frac{\alpha}{2} \, \Psi \left( x_{\alpha,\varepsilon}, y_{\alpha, \varepsilon} \right) \to 0$, as $M_{\alpha/2}$ and $M_\alpha$ converge to the same limit.
\end{proof}

\begin{proposition} \label{proposition:comparison_conditions_on_H}
Fix $\lambda >0$, $h \in C_b(E)$ and consider $u$ and $v$ sub- and super-solution to $f - \lambda Hf = h$.

\smallskip

Let $\{\Psi_\alpha\}_{\alpha > 0}$ be a family of good penalization functions and $\Upsilon$ be a good containment function. Moreover, for every $\alpha,\varepsilon >0$ let $x_{\alpha,\varepsilon},y_{\alpha,\varepsilon} \in E$ be such that
\begin{multline} \label{eqn:comparison_principle_proof_choice_of_sequences}
\frac{u(x_{\alpha,\varepsilon})}{1-\varepsilon} - \frac{v(y_{\alpha,\varepsilon})}{1+\varepsilon} -  \Psi_\alpha(x_{\alpha,\varepsilon},y_{\alpha,\varepsilon}) - \frac{\varepsilon}{1-\varepsilon}\Upsilon(x_{\alpha,\varepsilon}) -\frac{\varepsilon}{1+\varepsilon}\Upsilon(y_{\alpha,\varepsilon}) \\
= \sup_{x,y \in E} \left\{\frac{u(x)}{1-\varepsilon} - \frac{v(y)}{1+\varepsilon} - \Psi_\alpha(x,y)  - \frac{\varepsilon}{1-\varepsilon}\Upsilon(x) - \frac{\varepsilon}{1+\varepsilon}\Upsilon(y)\right\}.
\end{multline}

Suppose that
\begin{multline}\label{condH:negative:liminf}
\liminf_{\varepsilon \rightarrow 0} \liminf_{\alpha \rightarrow \infty} H\left(x_{\alpha,\varepsilon},\nabla \Psi_\alpha(\cdot,y_{\alpha,\varepsilon})(x_{\alpha,\varepsilon})\right) \\
- H\left(y_{\alpha,\varepsilon},\nabla \Psi_\alpha(\cdot,y_{\alpha,\varepsilon})(x_{\alpha,\varepsilon})\right) \leq 0,
\end{multline}
then $u \leq v$. In other words: $f - \lambda H f = h$ satisfies the comparison principle. 
\end{proposition}

\begin{proof}
By Lemma \ref{lemma:viscosity_extension} we get immediately that $u$ is a sub-solution to $f - \lambda H_\dagger f = h$ and $v$ is a super-solution to $f - \lambda H_\ddagger f = h$ . Thus, it suffices to verify the comparison principle for the equations involving the extensions $H_\dagger$ and $H_\ddagger$.

\smallskip

Let $x_{\alpha,\varepsilon},y_{\alpha,\varepsilon} \in E$ such that  \eqref{eqn:comparison_principle_proof_choice_of_sequences} is satisfied. Then, for all $\alpha$ we obtain that
\begin{align}
& \sup_x u(x) - v(x) \notag\\
& = \lim_{\varepsilon \rightarrow 0} \sup_x \frac{u(x)}{1-\varepsilon} - \frac{v(x)}{1+\varepsilon} \notag\\
& \leq \liminf_{\varepsilon \rightarrow 0} \sup_{x,y} \frac{u(x)}{1-\varepsilon} - \frac{v(y)}{1+\varepsilon} -  \Psi_\alpha(x,y) - \frac{\varepsilon}{1-\varepsilon} \Upsilon(x) - \frac{\varepsilon}{1+\varepsilon}\Upsilon(y) \notag\\
& = \liminf_{\varepsilon \rightarrow 0} \frac{u(x_{\alpha,\varepsilon})}{1-\varepsilon} - \frac{v(y_{\alpha,\varepsilon})}{1+\varepsilon} - \Psi_\alpha(x_{\alpha,\varepsilon},y_{\alpha,\varepsilon}) - \frac{\varepsilon}{1-\varepsilon}\Upsilon(x_{\alpha,\varepsilon}) -\frac{\varepsilon}{1+\varepsilon}\Upsilon(y_{\alpha,\varepsilon}) \notag \\
& \leq \liminf_{\varepsilon \rightarrow 0} \frac{u(x_{\alpha,\varepsilon})}{1-\varepsilon} - \frac{v(y_{\alpha,\varepsilon})}{1+\varepsilon}, \label{eqn:basic_inequality_on_sub_super_sol}
\end{align}
as $\Upsilon$ and $\Psi_\alpha$ are non-negative functions. Since $u$ is a sub-solution to $f - \lambda H_\dagger f = h$ and $v$ is a super-solution to $f - \lambda H_\ddagger f = h$, we find by our particular choice of $x_{\alpha,\varepsilon}$ and $y_{\alpha,\varepsilon}$ that
\begin{align}
& u(x_{\alpha,\varepsilon}) - \lambda H\left(x_{\alpha,\varepsilon}, (1-\varepsilon)\nabla \Psi_\alpha(\cdot,y_{\alpha,\varepsilon})(x_{\alpha,\varepsilon}) + \varepsilon \nabla \Upsilon(x_{\alpha,\varepsilon})\right) \leq h(x_{\alpha,\varepsilon}), \label{eqn:ineq_comp_proof_1}\\
& v(y_{\alpha,\varepsilon}) - \lambda H\left(y_{\alpha,\varepsilon},-(1+\varepsilon)\nabla \Psi_\alpha(x_{\alpha,\varepsilon},\cdot)(y_{\alpha,\varepsilon}) - \varepsilon \nabla \Upsilon(y_{\alpha,\varepsilon})\right) \geq h(y_{\alpha,\varepsilon}).\label{eqn:ineq_comp_proof_2}
\end{align}
For all $z \in E$, the map $p \mapsto H(z,p)$ is convex. Thus, \eqref{eqn:ineq_comp_proof_1} implies that
\begin{multline} \label{eqn:ineq_comp_proof_1_convexity}
u(x_{\alpha,\varepsilon}) \leq h(x_{\alpha,\varepsilon}) + (1-\varepsilon) \lambda H(x_{\alpha,\varepsilon}, \nabla \Psi_\alpha(\cdot,y_{\alpha,\varepsilon})(x_{\alpha,\varepsilon}))  \\
+ \varepsilon \lambda H(x_{\alpha,\varepsilon},\nabla \Upsilon(x_{\alpha,\varepsilon})).
\end{multline}
For the second inequality, first note that because $\Psi_\alpha$ are good penalization functions, we have $- ( \nabla \Psi_\alpha(x_{\alpha,\varepsilon},\cdot))(y_{\alpha,\varepsilon}) = \nabla \Psi_\alpha(\cdot, y_{\alpha,\varepsilon})(x_{\alpha,\varepsilon})$. Next, we need a more sophisticated bound using the convexity of $H$:
\begin{align*}
& H(y_{\alpha,\varepsilon}, \nabla \Psi_\alpha(\cdot, y_{\alpha,\varepsilon})(x_{\alpha,\varepsilon})) \\
& \leq \frac{1}{1+\varepsilon} H(y_{\alpha,\varepsilon},(1+\varepsilon)\nabla \Psi_\alpha(\cdot, y_{\alpha,\varepsilon})(x_{\alpha,\varepsilon}) - \varepsilon \nabla \Upsilon(y_{\alpha,\varepsilon})) + \frac{\varepsilon}{1+\varepsilon} H(y_{\alpha,\varepsilon}, \nabla \Upsilon(y_{\alpha,\varepsilon})).
\end{align*}
Thus, \eqref{eqn:ineq_comp_proof_2} gives us
\begin{equation} \label{eqn:ineq_comp_proof_2_convexity}
v(y_{\alpha,\varepsilon}) \geq h(y_{\alpha,\varepsilon}) + \lambda (1+\varepsilon) H(y_{\alpha,\varepsilon},\nabla\Psi_\alpha(\cdot,y_{\alpha,\varepsilon})(x_{\alpha,\varepsilon})) - \varepsilon \lambda H(y_{\alpha,\varepsilon},\nabla \Upsilon(y_{\alpha,\varepsilon})).
\end{equation}
By combining \eqref{eqn:basic_inequality_on_sub_super_sol} with \eqref{eqn:ineq_comp_proof_1_convexity} and \eqref{eqn:ineq_comp_proof_2_convexity}, we find
\begin{align} 
& \sup_x u(x) - v(x) \nonumber\\
& \leq \liminf_{\varepsilon \rightarrow 0} \liminf_{\alpha \rightarrow \infty} \left\{ \frac{h(x_{\alpha,\varepsilon})}{1 - \varepsilon} - \frac{h(y_{\alpha,\varepsilon})}{1+\varepsilon} \right.  \label{eqn:eqn:comp_proof_final_bound:line1}\\
& \qquad + \frac{\varepsilon}{1-\varepsilon}H(x_{\alpha,\varepsilon}, \nabla \Upsilon(x_{\alpha,\varepsilon})) + \frac{\varepsilon}{1+\varepsilon}H(y_{\alpha,\varepsilon}, \nabla\Upsilon(y_{\alpha,\varepsilon})) \label{eqn:eqn:comp_proof_final_bound:line2}\\
& \left. \qquad +  \lambda \left[H(x_{\alpha,\varepsilon},\nabla\Psi_\alpha(\cdot,y_{\alpha,\varepsilon})(x_{\alpha,\varepsilon})) - H(y_{\alpha,\varepsilon},\nabla\Psi_\alpha(\cdot,y_{\alpha,\varepsilon})(x_{\alpha,\varepsilon}))\right] \vphantom{\sum} \right\}.\label{eqn:eqn:comp_proof_final_bound:line3}
\end{align}
The term \eqref{eqn:eqn:comp_proof_final_bound:line3} vanishes by assumption. Now observe that, for fixed $\varepsilon$ and varying $\alpha$, the sequence $(x_{\alpha,\varepsilon},y _{\alpha,\varepsilon})$ takes its values in a compact set and, hence, admits converging subsequences. All these subsequences converge to points of the form $(z,z)$. Therefore, as $\alpha \rightarrow \infty$, we find
\[
\liminf_{\varepsilon \rightarrow 0} \liminf_{\alpha \rightarrow \infty}  \frac{h(x_{\alpha,\varepsilon})}{1 - \varepsilon} - \frac{h(y_{\alpha,\varepsilon})}{1+\varepsilon} \leq \liminf_{\varepsilon \rightarrow 0} \vn{h} \frac{2\varepsilon}{1-\varepsilon^2} = 0,
\]
giving that also the term in \eqref{eqn:eqn:comp_proof_final_bound:line1} converges to zero. The term in \eqref{eqn:eqn:comp_proof_final_bound:line2} vanishes as well, due to the uniform bounds on $H(z,\nabla \Upsilon(z))$ by property ($\Upsilon$d).

\smallskip

We conclude that the comparison principle holds for $f - \lambda H f = h$.
\end{proof}

\subsection{Compact containment and the large deviation principle} \label{section:appendix_ldp}

To connect the Hamilton-Jacobi equation to the large deviation principle, we introduce some additional concepts. Fix some $d \geq 1$. In this section, we assume that $E$ is a \textit{closed} subset of $\bR^d$ that is contained in the $\bR^d$-closure of its $\bR^d$-interior. Additionally, we have closed subspaces $E_n \subseteq E$ for all $n$ and assume that $E = \lim_{n \to \infty} E_n$, i.e. for every $x \in E$ there exist $x_n \in E_n$ such that $x_n \rightarrow x$. We consider the following notion of operator convergence. 

\begin{definition}
Suppose that for each $n$ we have an operator $(B_n,\cD(B_n))$, $B_n : \cD(B_n) \subseteq C_b(E_n) \rightarrow C_b(E_n)$. The \textit{extended limit} $ex-\lim_n B_n$ is defined by the collection $(f,g) \in C_b(E) \times C_b(E)$ such that there exist $f_n \in \cD(B_n)$ satisfying
\begin{equation} \label{eqn:convergence_condition}
\lim_{n \rightarrow \infty} \sup_{x \in K \cap E_n} \left|f_n(x) - f(x)\right| + \left|B_n f_n(x) - g(x)\right| = 0.
\end{equation}
For an operator $(B,\cD(B))$, we write $B \subseteq ex-\lim_n B_n$ if the graph $\{(f,Bf) \, | \, f \in \cD(B) \}$ of $B$ is a subset of $ex-\lim_n B_n$.
\end{definition}

\begin{remark}
The notion of extended limit can be generalized further, e.g. for limiting spaces like in the previous two sections. Such abstract generalizations are carried out in Definition 2.5 and Condition 7.11 of \cite{FK06}. Our definition for a \textit{closed} limiting space $E$ is the simplest version of this abstract machinery. 
\end{remark}

\begin{assumption} \label{assumption:LDP_assumption}
Fix some $d \geq 1$. Let $E$ be a closed subset of $\bR^d$ that is contained in the $\bR^d$-closure of its $\bR^d$-interior and let $E_n$ be closed subsets of $E$ such that $E = \lim_{n \rightarrow \infty} E_n$. 

Assume that for each $n \geq 1$, we have $A_n \subseteq C_b(E_n) \times C_b(E_n)$ and existence and uniqueness holds for the $D_{E_n}(\bR^+)$ martingale problem for $(A_n,\mu)$ for each initial distribution $\mu \in \cP(E_n)$. Letting $\PR_{y}^n \in \cP(D_{E_n}(\bR^+))$ be the solution to $(A_n,\delta_y)$, the mapping $y \mapsto \PR_y^n$ is measurable for the weak topology on $\cP(D_{E_n}(\bR^+))$. Let $X_n$ be the solution to the martingale problem for $A_n$ and set
\begin{equation*}
H_n f = \frac{1}{r(n)} e^{-r(n)f}A_n e^{r(n)f} \qquad e^{r(n)f} \in \cD(A_n),
\end{equation*}
for some sequence of speeds $\{r(n)\}_{n \geq 1}$, with $\lim_{n \rightarrow \infty} r(n) = \infty$.

Suppose that we have an operator $H : \cD(H) \subseteq C_b(E) \rightarrow C_b(E)$ with $\cD(H) = C^2_c(E)$ of the form $Hf(x) = H(x,\nabla f(x))$ which satisfies $H \subseteq ex-\lim H_n$.

%Suppose $H$ is of the form $Hf(x) = H(x,\nabla f(x))$ and that for every $z \in E$, the map $p \mapsto H(z,p)$ is convex.
\end{assumption}

\begin{proposition} \label{proposition:exp_compact_containment}
Suppose Assumption \ref{assumption:LDP_assumption} is satisfied and assume that $\Upsilon$ is a good containment function. Suppose that the sequence $\{X_n(0)\}_{n \geq 1}$ is exponentially tight with speed $\{r(n)\}_{n \geq 1}$.

Then the sequence $\{X_n\}_{n \geq 1}$ satisfies the exponential compact containment condition with speed $\{r(n)\}_{n \geq 1}$: for every $T > 0$ and $a \geq 0$, there exists a compact set $K_{a,T} \subseteq E$ such that
\begin{equation*}
\limsup_{n \rightarrow \infty} \frac{1}{r(n)} \log \PR\left[X_n(t) \notin K_{a,T} \text{ for some } t \leq T \right] \leq -a.
\end{equation*}
\end{proposition}

In the proof of this proposition, we apply Lemma 4.22 from \cite{FK06}. We recall it here for the sake of readability.

\begin{lemma}[Lemma 4.22 in \cite{FK06}] \label{lemma:compact_containment_FK}
Let $X_n$ be solutions of the martingale problem for $A_n$ and suppose that $\{X_n(0)\}_{n \geq 1}$ is exponentially tight with speed $\{r(n)\}_{n \geq 1}$. Let $K$ be compact and let $G \supseteq K$ be open. For each $n$, suppose we have $(f_n,g_n) \in H_n$. Define
\begin{equation*}
\beta(K,G) := \liminf_{n \rightarrow \infty} \left( \inf_{x \in G^c} f_n(x) - \sup_{x \in K} f_n(x)\right) \mbox{ and } \gamma(G) := \limsup_{n \rightarrow \infty} \sup_{x \in G} g_n(x).
\end{equation*}
Then
\begin{multline} \label{eqn:compact_containment_bound}
\limsup_{n \rightarrow \infty} \frac{1}{r(n)} \log \PR\left[X_n(t) \notin G \text{ for some } t \leq T \right] \\
\leq \max \left\{-\beta(K,G) + T \gamma(G), \limsup_{n\rightarrow \infty} \PR\left[X_n(0) \notin K\right] \right\}.
\end{multline}
\end{lemma}

Note that in the case the closure of $G$ is compact, this result is suitable for proving the compact containment condition. This is what we will use below.

\begin{proof}[Proof of Proposition \ref{proposition:exp_compact_containment}]
Fix $a \geq 0$ and $T > 0$. We will construct a compact set $K'$ such that
\begin{equation*}
\limsup_{n \rightarrow \infty} \frac{1}{r(n)} \log \PR\left[X_n(t) \notin K' \text{ for some } t \leq T \right] \leq -a.
\end{equation*}
As $X_n(0)$ is exponentially tight with speed $\{r(n)\}_{n \geq 1}$, we can find a sufficiently large $R \geq 0$ so that 
\begin{equation*}
\limsup_{n\rightarrow \infty} \frac{1}{r(n)} \log \PR\left[X_n(0) \notin \overline{B(x_0,R)}\right] \leq - a,
\end{equation*}
where $x_0$ is a point such that $\Upsilon(x_0) = 0$ and $\overline{B(x_0,R)}$ is the closed ball with radius $R$  and center $x_0$. Thus, by Lemma \ref{lemma:compact_containment_FK} it suffices to find $(f_n,g_n) \in H_n$, a compact $K$ and an open set $G$ such that $-\beta(K,G) + T\gamma(G) \leq - a$.

\smallskip

Set $\gamma := \sup_{z} H(z,\nabla \Upsilon(z))$ and $c_1 := \sup_{z \in \overline{B(x_0,R)}} \Upsilon(z)$. Observe that $\gamma  < \infty$ by assumption ($\Upsilon$d) and $c_1 < \infty$ by compactness. Now choose $c_2$ such that
\begin{equation} \label{eqn:proof_compact_containment_choice_c2}
-[c_2 - c_1] + T\gamma = -a
\end{equation}
and take $K = \{z \in E \, | \, \Upsilon(z) \leq c_1\}$ and $G = \{z \in E \, | \, \Upsilon(z) < c_2\}$.

\smallskip

Let $\theta : [0,\infty) \rightarrow [0,\infty)$ be a compactly supported smooth function with the property that $\theta(z) = z$ for $z \leq c_2$. For each $n$, define $f_n := \theta \circ \Upsilon$ and $g_n := H_n f_n$. By Assumption \ref{assumption:LDP_assumption}, $g_n \rightarrow H f$ in the sense of \eqref{eqn:convergence_condition} and moreover, by construction $\beta(K,G) = c_2 - c_1$ and $\gamma(G) = \gamma$. Thus by \eqref{eqn:proof_compact_containment_choice_c2} and Lemma \ref{lemma:compact_containment_FK} we obtain
\begin{equation*}
\limsup_{n \rightarrow \infty} \frac{1}{r(n)} \log \PR\left[X_n(t) \notin G \text{ for some } t \leq T \right] \leq -a
\end{equation*}
and the compact containment condition holds with $K_{a,T} = \overline{G}$.
\end{proof}

\begin{theorem}[Large deviation principle] \label{theorem:Abstract_LDP}
Suppose Assumption \ref{assumption:LDP_assumption} is satisfied and assume that $\Upsilon$ is a good containment function for $H$. % and that $\{\Psi_\alpha\}_{\alpha > 0}$ are good penalizations functions. 
Then we have the following result.

Suppose that for all $\lambda > 0$ and $h \in C_b(E)$ the comparison principle holds for $f - \lambda H f = h$. And suppose that the sequence $\{X_n(0)\}_{n \geq 1}$ satisfies the large deviation principle with speed $\{r(n)\}_{n \geq 1}$ on $E$ with good rate function $I_0$. 

Then the large deviation principle with speed $\{r(n)\}_{n \geq 1}$ holds for $\{X_n\}_{n \geq 1}$ on $D_E(\bR^+)$ with good rate function $I$. Additionally, suppose that the map $p \mapsto H(x,p)$ is convex and differentiable for every $x$ and that the map $(x,p) \mapsto \frac{\dd}{\dd p} H(x,p)$ is continuous. Then the rate function $I$ is given by
\begin{equation*}
I(\gamma) = \begin{cases}
I_0(\gamma(0)) + \int_0^\infty \cL(\gamma(s),\dot{\gamma}(s)) \dd s & \text{if } \gamma \in \cA\cC, \\
\infty & \text{otherwise},
\end{cases}
\end{equation*}
where $\cL : E \times \bR^d \rightarrow \bR$ is defined by $\cL(x,v) = \sup_p \ip{p}{v} -H(x,p)$.
\end{theorem}

\begin{proof}
The large deviation result follows from Theorem 7.18 in \cite{FK06}. Referring to the notation therein, we are using $H_\dagger = H_\ddagger = H$.

\smallskip

The representation of the rate function in terms of the Lagrangian can be carried out by using Theorem 8.27 and Corollary 8.28 in \cite{FK06}. See for example the Section 10.3 in \cite{FK06} for this representation in the setting of Freidlin-Wentzell theory. Alternatively, an application of this result in a compact setting has been carried out also in \cite{Kr16b}. The only extra condition compared to Theorem 6 in \cite{Kr16b} due to the non-compactness is the verification of Condition 8.9.(4) in \cite{FK06}:

\smallskip

\textit{For each compact $K \subseteq E$, $T > 0$ and $0 \leq M < \infty$, there exists a compact set $\hat{K} = \hat{K}(K,T,M) \subseteq E$ such that if $\gamma \in \cA\cC$ with $\gamma(0) \in K$ and
\begin{equation} \label{eqn:bound_on_lagrangian_cost}
\int_{0}^T \cL(\gamma(s),\dot{\gamma}(s)) \dd s \leq M,
\end{equation}
then $\gamma(t) \in \hat{K}$ for all $t \leq T$.}

\smallskip

This condition can be verified as follows. Recall that the level sets of $\Upsilon$ are compact. Thus, we control the growth of $\Upsilon$. Let $\gamma \in \cA\cC$ satisfy the conditions given above. Then
\begin{align*}
\Upsilon(\gamma(t)) & = \Upsilon(\gamma(0)) + \int_0^t \ip{\nabla\Upsilon(\gamma(s))}{\dot{\gamma}(s)} \dd s \\
& \leq \Upsilon(\gamma(0)) + \int_0^t \cL(\gamma(s),\dot{\gamma}(s)) + H(\gamma(s),\nabla \Upsilon(\gamma(s))) \dd s \\
& \leq \sup_{y \in K} \Upsilon(y) + M + \int_0^T \sup_z  H(z,\nabla \Upsilon(z)) \dd s \\
& =: C < \infty.
\end{align*}
Thus, we can take $\hat{K} = \{z \in E \, | \, \Upsilon(z) \leq C\}$.
\end{proof}

\smallskip

\textbf{Acknowledgement}
The authors are supported by The Netherlands Organisation for Scientific Research (NWO): RK via grant 600.065.130.12N109 and FC via TOP-1 grant 613.001.552.

\printbibliography
%\bibliographystyle{abbrv} 
%\bibliography{KraaijBib,ColletBib}{}

\end{document}